\documentclass{article}     

\usepackage{amsfonts}
\usepackage{amsmath}
\usepackage{graphicx}
\usepackage{amssymb}
\usepackage{amsthm}
\usepackage{hyperref}
\usepackage[all]{xy}
\usepackage[english]{babel}

\title{Bordered Floer Homology of (2,2n)-Torus Link Complement}

\author{Jaepil Lee}

\newtheorem{thm}{Theorem}[section]    
\newtheorem{lem}{Lemma}[section]      
\newtheorem{prop}[thm]{Proposition}
\theoremstyle{definition}
\newtheorem{defn}[thm]{Definition}    
\newtheorem{rem}[thm]{Remark}             

\makeatletter
  \let\c@lem=\c@thm
\makeatother

\begin{document}

\maketitle

\begin{abstract}   
We compute the bordered Floer homology $\widehat{CFDD}$ of (2,2n)-torus link complement and discuss assorted examples and type-$DD$ structure homotopy equivalence.
\end{abstract}

\section{Introduction}

In recent years, Heegaard Floer theory has fascinated many low-dimensional topologists. Developed by  P. Ozsv\'ath and Z. Sz\'abo, Heegaard Floer invariants of closed three-manifolds led to a breakthrough in low dimensional topology. These invariants were recently shown to be equivalent to three-dimensional Seiberg-Witten Floer homology by Kutluhan, Lee and Taubes \cite{KLT}. They were also proven to be equivalent to contact homology by Colin, Ghiggini and Honda \cite{CGH}; this equivalence had initially motivated Oszv\'ath--Szab\'o's constructions. Moreover, Heegaard Floer theory turned out to be useful in defining knot and link invariants. See Ozsv\'ath--Sz\'abo \cite{OZ04, OZ05}, and Rasmussen \cite{Ras02}. These invariants are now known as \emph{knot Floer homology} and \emph{link Floer homology}. In particular, knot Floer homology and Heegaard Floer homology of a three-manifold obtained by integral surgery on knot turned out to be closely related. See Rasmussen \cite{Ras02} and Ozsv\'ath--Sz\'abo \cite{OZ08}. For the link surgery case, the relation was discovered but appeared more complicated than the knot case. See Manolescue--Ozsv\'ath \cite{MO10}. \\

More recently, Lipshitz, Ozsv\'ath and Thurston extended the theory to three-manifolds with nonempty boundary. \emph{Bordered Floer homology}, first introduced in ~\cite{LOT08}, consists of two different modules: $\widehat{CFD}$ and $\widehat{CFA}$. The homotopy type of each module is a topological invariant of a three-manifold with \emph{connected} boundary equipped with a framing (a diffeomorphism to a model surface). The bordered theory is a powerful tool thanks to the pairing theorem:  one can recover the Heegaard Floer homology of a closed 3-manifold decomposed into two pieces by taking ``$\mathcal{A}_{\infty}$-tensor product'' of $\widehat{CFA}$ of the first piece and $\widehat{CFD}$ of the second piece. \\

Bordered Floer homology of a three-manifold with genus one boundary surface is related to knot Floer homology. In \cite[Chapter 11]{LOT08}, they described an algorithm to recover $\widehat{CFD}(S^3 \backslash \nu(K))$ (with arbitrary framing) from knot Floer homology $CFK^-(K)$. This makes it possible to compute the Heegaard Floer homology of the surgered manifold by taking $\mathcal{A}_{\infty}$-tensor product with the solid torus. \\

Moreover, Lipshitz, Ozsv\'ath and Thurston have generalized bordered Floer homology to \emph{doubly bordered Floer homology} \cite{LOT11}. As the name suggests, this is an invariant associated to a three-manifold with two boundary components; we get three different types of bimodules, $\widehat{CFDA}$, $\widehat{CFDD}$, and $\widehat{CFAA}$. An important example of a manifold with two boundary components is $S^3 \backslash \nu(L)$, a complement of a two-component link $L$. \\

In this paper, we give a  calculation of $\widehat{CFDD}(S^3 \backslash \nu(L))$, where $L$ is $(2,2n)$-torus link. For a number of reasons, we mainly focus on the $\widehat{CFDD}$ module. First, it is the easiest bimodule to compute since it does not involve any $\mathcal{A}_{\infty}$ structure. Second, it is always possible to convert $\widehat{CFDD}$ module to $\widehat{CFDA}$ or $\widehat{CFAA}$, by attaching $\widehat{CFAA}(\mathbb{I})$ module to the left or right side of the $\widehat{CFDD}$ module. In Section~\ref{sec:prelim}, we collect the necessary background and notation. The actual calculation is in Section~\ref{sec:main}; the answer is shown in Proposition~\ref{thm:main}. (See also Figure~\ref{fig:26differential} for a $(2,6)$-torus link case.) The simplified version of the answer is in Figure~\ref{fig:simplified01}. We work with a specific Heegaard diagram in order to find the generators and  differentials of the module explicitly. However, only a few of the differentials can be obtained by the  direct examination of their domains; for the remaining differentials, we have to exploit the $A_\infty$-structure of $\widehat{CFAA}$. In Section~\ref{sec:example}, we give several applications of the pairing formula, recovering some known Floer homologies from our calculation, to illustrate and check the result. \\


\textbf{Acknowledgement} The author would like to thank Robert Lipshitz and Peter Ozsv\'ath for very helpful conversation and advice, and Adam Levine for commenting on the final version of this paper. Lastly, I give thanks to my advisor Olga Plamenevskaya for her enormous patience and encouragement.

\section{Background on Doubly Bordered Floer Theory}
\label{sec:prelim}

We will assume that the reader is familiar with bordered Floer homology of a single boundary case. If not, we suggest the reader refer to~\cite{LOT11b} for a brief introduction to the topic. In this section, we merely list the definitions and cursory reviews that will be used in the rest of the paper. We briefly recall algebraic preliminaries of $\mathcal{A}_{\infty}$-algebras and assorted modules used in the original bordered Floer homology~\cite{LOT08}, and generalize the result to focus especially on the doubly bordered Floer homology introduced by Lipshitz, Ozsv\'ath and Thurston \cite{LOT11}. \\

\subsection{Algebraic Preliminaries}
We first begin with the algebra associated to a boundary surface of a three-manifold. In a handle decomposition of a genus $g$ surface $\Sigma^g$, the zero-handle $D$ of $\Sigma^g$ has $2g$ marked points $\mathbf{a}$ on $\partial D = Z$ equipped with a two-to-one matching $M$ between the points so that each one-handle is attached to a pair of matched points. We also fix a point $z$ on $Z$ away from $\mathbf{a}$. This set of data is called a \emph{pointed matched circle} and denoted by $\mathcal{Z} = \{ Z, \mathbf{a}, M, z \}$. $F( \mathcal{Z} )$ denotes the surface obtained by the data and $D \subset F( \mathcal{Z} )$ is called a \emph{preferred disk}. The bordered Floer package associates a $dg$ algebra to $\mathcal{Z}$, which will be a \emph{strands algebra}, and denoted by $\mathcal{A} (\mathcal{Z})$. \\

Since we will be studying the torus boundary case, from now on we will assume that the genus $g$ of the boundary surface equals one. If $g=1$, $\mathcal{A} (\mathcal{Z})$ is $\mathbb{F}_2$-vector space generated by Reeb chords respecting the boundary orientation of $Z =\partial D$. In other words, it is generated by $\rho_I$, $I \in \{ 1, 2, 3, 12, 23, 123 \}$ and two idempotents $\iota_1$ and $\iota_2$ such that $\iota_1 + \iota_2 =1$. The multiplication rule between Reeb chords follows the concatenation rule of labels of chords. For example, $\rho_1 \cdot \rho _{23} = \rho_{123}$. If the labels of two Reeb chords do not match, then it is zero. $\mathcal{I} \subset \mathcal{A}(\mathcal{Z})$ denotes the subalgebra generated by idempotents $\iota_1$ and $\iota_2$. This strands algebra is called a \emph{torus algebra}. A detailed description can be found in Lipshitz, Ozsv\'ath and Thurston \cite{LOT08}. \\

Throughout this section, we will be referring $\mathbb{F}$ to $\mathbb{F}_2$. \\

Next, we will study a (right) \emph{$\mathcal{A}_{\infty}$-module} and a (left) \emph{type-$D$ module}. For an $\mathcal{A}_{\infty}$-algebra $(A, \mu_i)$, an $\mathcal{A}_{\infty}$-module is a $\mathbb{F}$-module $M$, equipped with maps
\begin{displaymath}
m_i : M \otimes A^{\otimes (i-1)} \rightarrow M
\end{displaymath}
satisfying compatibility relations
\begin{eqnarray*}
0 & = & \sum_{i+j=n+1} m_i (m_j ( \mathbf{x} \otimes a_1 \otimes \cdots \otimes a_{j-1} ) \otimes \cdots \otimes a_{n-1} ) \\
 & + & \sum_{i+j=n+1} \sum_{l=1}^{n-j} m_i ( \mathbf{x} \otimes a_1 \otimes \cdots a_{l-1} \otimes \mu_j (a_l \otimes \cdots \otimes a_{l+j-1} ) \otimes \cdots \otimes a_{n-1} ),
\end{eqnarray*}
for all $i \geq 1$. An $\mathcal{A}_{\infty}$-module is \emph{strictly unital} if $m_2( \mathbf{x} \otimes 1) = \mathbf{x}$ and $m_i ( \mathbf{x} \otimes a_1 \otimes \cdots \otimes a_{i-1} )=0$ for $i > 2$ and some $a_j =1$. \\

In bordered Floer theory, an $\mathcal{A}_{\infty}$-module is called \emph{type-$A$} module. \\ 

For a $dg$-algebra $(A, \mu_1, \mu_2)$, a type-$D$ module is a $\mathbb{F}$-module equipped with a map $\delta^1 : N \rightarrow A \otimes N$, satisfying the compatibility relation
\begin{displaymath}
0 = (\mu_2 \otimes \mathbb{I}_N ) \circ ( \mathbb{I}_A \otimes \delta^1 ) \circ \delta^1 + (\mu_1 \otimes \mathbb{I}_N ) \circ \delta^1.
\end{displaymath}

These modules are generalized to the following bimodules, namely a \emph{type-$AA$ bimodule} and a \emph{type-$DD$ bimodule}. In this paper, we will be mainly studying these bimodules.  

\begin{defn}
  Let $A$ and $B$ be $\mathcal{A}_{\infty}$-algebras over $\mathbb{F}$ equipped with $\mathcal{A}_{\infty}$-maps $\{ \mu^A_i \}$ and $\{ \mu^B_i \}$, respectively. A \emph{right-right $\mathcal{A}_{\infty}$-bimodule} of \emph{type-$AA$ bimodule} $M_{\mathcal{A}, \mathcal{B}}$ over $A$ and $B$ consists of a right-right $(\mathbb{F}, \mathbb{F})$-bimodule $M$ and maps   
  \begin{displaymath}
    m_{1,i,j} : M \otimes A^{\otimes i} \otimes B^{\otimes j} \rightarrow M
  \end{displaymath}
  such that the following compatibility condition holds.
  \begin{eqnarray*}
    0=\sum_{\substack{ k+l = i+1 \\ \lambda + \eta = j+1}}
    m_{1,k,\lambda} ( m_{1,l,\eta} ( \mathbf{x}, a_1 \otimes
    \cdots \otimes a_{l-1}, b_1 \otimes \cdots \otimes
    b_{\lambda -1} ), a_l \otimes \cdots \otimes a_{i-1},
    b_{\lambda} \otimes \cdots \otimes b_{j-1} ) \\
    + \sum_{k+l = i+1} \sum_{n=1}^{i-l} m_{1,k,j} ( \mathbf{x},
    a_1 \otimes \cdots \otimes a_{n-1} \otimes \mu^A_l ( a_n
    \otimes \cdots \otimes a_{n+l-1} ) \otimes \cdots \otimes
    a_{i-1} , b_1 \otimes \cdots \otimes b_{j-1} ) \\
    + \sum_{\lambda+\eta = j+1} \sum_{n=1}^{j-\eta} m_{1,i,\lambda}
    (\mathbf{x}, a_1 \otimes \cdots \otimes a_{i-1} ,
    b_1 \otimes \cdots \otimes b_{n-1} \otimes \mu^B_l ( b_n
    \otimes \cdots \otimes b_{n+l-1} ) \otimes \cdots \otimes
    b_{j-1} ) \\
  \end{eqnarray*}
  for all $i$ and $j$. 
\end{defn}  
By writing $m = \sum_{i,j} m_{1,i,j}$, the compatibility condition can be drawn as the diagram below.
\begin{displaymath}
  \xymatrix @-1pc {
    \ar@{-->}[dd] & \ar[d] & \ar[d] & & \ar@{-->}[ddd] & \ar[d] & \ar@/^1pc/[dddll] & & \ar@{-->}[ddd] & \ar@/^/[dddl] & \ar[d] \\
    & \Delta^A \ar[dl] \ar[ddl] & \Delta^B \ar@/^/[dll] \ar@/^/[ddll] & & & \overline{D}^A \ar[ddl] & & & & & \overline{D}^B \ar[ddll] & \\
    m \ar@{-->}[d] & & & + & & & & + & & & & = 0 \\
    m \ar@{-->}[d] & & & & m \ar@{-->}[d] & & & & m \ar@{-->}[d] & & & \\
    & & & & & & & & & & &
  }
\end{displaymath}
The dashed line above represents a module element, and the regular line represents an element from tensor algebra $\mathcal{T}^* A$ and $\mathcal{T}^* B$. The map $\Delta^A : \mathcal{T}^* A \rightarrow \mathcal{T}^* A \otimes \mathcal{T}^* A$ represents the canonical comultiplication
\begin{displaymath}
  \Delta^A(a_1 \otimes \cdots \otimes a_n) = \sum_{m=0}^n (a_1 \otimes \cdots \otimes a_m) \otimes (a_{m+1} \otimes \cdots \otimes a_n),
\end{displaymath}
and $\overline{D}^A : \mathcal{T}^* A \rightarrow \mathcal{T}^* A$ is defined as
\begin{displaymath}
\overline{D}^A (a_1 \otimes \cdots \otimes a_n ) = \sum_{j=1}^n \sum_{l=1}^{n-j+1} a_1 \otimes \cdots \otimes \mu_j^A ( a_l \otimes \cdots \otimes a_{l+j-1}) \otimes \cdots \otimes a_n.
\end{displaymath}
$\Delta^B$ and $\overline{D}^B$ are defined similarly. \\

\begin{defn}
Let $A$ and $B$ be $\mathcal{A}_{\infty}$-algebras over $\mathbb{F}$. A \emph{left-left type-$DD$ bimodule} ${}^{A, B} M$ over $A$ and $B$ consists of left-left $(\mathbb{F}, \mathbb{F})$-bimodule $M$ and maps
  \begin{displaymath}
    \delta^1 : M \rightarrow A \otimes B \otimes M 
  \end{displaymath}
satisfying the following compatibility condition.
  \begin{displaymath}
    ( ( \mu_2^L , \mu_2^R ) \otimes \mathbb{I}_M ) \circ ( (
    \mathbb{I}_A , \mathbb{I}_B ) \otimes \delta^1 ) \circ \delta^1
    + ( (\mu_1^L , \mathbb{I}_B) \otimes \mathbb{I}_M ) \circ
    \delta^1 + ( (\mathbb{I}_A, \mu_1^R) \otimes \mathbb{I}_M )
    \circ \delta^1 = 0.
  \end{displaymath}
\end{defn}
Again, the compatibility condition is drawn as the diagram below.
\begin{displaymath}
  \xymatrix @-1pc {
    & & \ar@{-->}[d] & & & & \ar@{-->}[d] & & & & \ar@{-->}[d] & \\
    & & \delta^1 \ar@{-->}[d] \ar@/_/[ddl] \ar@/_/[ddll] & & & & \delta^1 \ar@{-->}[ddd] \ar@/_/[ddl] \ar@/_1pc/[dddll] & & & & \delta^1 \ar@{-->}[ddd] \ar@/_1pc/[ddll] \ar@/_/[dddl] & \\
    & & \delta^1 \ar@{-->}[dd] \ar[dl] \ar[dll] & + & & & & + & & & & = 0 \\
    \mu_2 \ar[d] & \mu_2 \ar[d] & & & & \mu_1 \ar[d] & & & \mu_1 \ar[d] & & & \\
    & & & & & & & & & & &
  }
\end{displaymath}

\subsection{Heegaard diagram of the bordered three-manifold}

A \emph{bordered three-manifold} is a quadruple $(Y_1, \Delta_1, z_1, \psi_1)$, where $Y_1$ is a three-manifold with boundary, $\Delta_1$ is a disk in $\partial Y_1$, $z_1$ is a point in $\partial \Delta_1$, and $\psi_1 : ( F( \mathcal{Z} ), D, z ) \rightarrow (\partial Y_1, \Delta_1, z_1 ) $ is a parametrization of boundary. That is , $\psi$ is a homeomorphism from $F (\mathcal{Z})$ to $\partial Y_1$ sending $D$ to $\Delta_1$ and $z$ to $z_1$. \\

To describe a bordered three-manifold, we use a \emph{bordered Heegaard diagram}.

\begin{defn}
A \emph{bordered Heegaard diagram} is a quadruple $\mathcal{H} = ( \overline{\Sigma}, \overline{\boldsymbol{\alpha}}, \boldsymbol{\beta}, z )$ consisting of
\begin{itemize}
  \item a compact, oriented surface $\overline{\Sigma}$ of genus $g$ with a single boundary component;
  \item a $g$-tuple of disjoint circles $\boldsymbol{\beta} = \{ \beta_1, \cdots, \beta_g \}$ in the interior of $\overline{\Sigma}$;
  \item a $g+k$-tuple of disjoint curves $\overline{ \boldsymbol{\alpha} } = \boldsymbol{\alpha}^c \cup \boldsymbol{\alpha}^a$ in $\overline{\Sigma}$, where $\boldsymbol{\alpha}^c = \{ \alpha^c_1, \cdots \alpha^c_{g-k} \}$ is a set of circles in the interior of $\overline{\Sigma}$, and $\boldsymbol{\alpha}^a = \{ \alpha^a_1, \cdots, \alpha^a_{2k} \}$ is a set of arcs whose boundaries are in $\partial \overline{\Sigma}$;
  \item a point $z$ in $\partial \overline{\Sigma}$, away from the boundaries of arcs in $\boldsymbol{ \alpha }^a$,
\end{itemize}
such that $\overline{\Sigma} \backslash \overline{ \boldsymbol{\alpha} }$ and $\overline{\Sigma} \backslash \boldsymbol{ \beta }$ are connected, and $\overline{ \boldsymbol{\alpha} }$ and $\boldsymbol{ \beta }$ intersect transversally.
\end{defn}
We construct a bordered three-manifold from a bordered Heegaard diagram $\mathcal{H}$ in the following manner. First, we obtain a three-manifold with boundary $Y(\mathcal{H})$ by thickening $\overline{\Sigma} \times [0,1]$ and attaching a three-dimensional two-handle to each $\alpha^c_i \times \{ 0 \} \times \overline{\Sigma}$ and a three-dimensional two-handle to each $\beta_i \times \{ 1 \} \times \overline{\Sigma}$. The boundary of the resulting manifold is a genus $k$ surface, and the surface is decomposed into a disk $D$ and a genus $k$ surface with a single boundary by $\partial \overline{\Sigma} \times \{ 1 \} $. Then, we get a bordered three-manifold $( Y( \mathcal{H}), D, z, \psi )$, where $\psi$ is determined by $\boldsymbol{\alpha}^a$, which is considered as a parametrization data of the surface. \\

A bordered Floer package defines a type-$D$ module $\widehat{CFD}( \mathcal{H})$ and a type-$A$ module $\widehat{CFA}( \mathcal{H} )$ from a bordered Heegaard diagram $\mathcal{H}$, which are well defined up to quasi-isomorphism. Each module has a generating set $\mathfrak{S} (\mathcal{H} )$, whose element $\mathbf{x} = \{ x_1, \cdots, x_g \} $ is a $g$-tuple of points in $\overline{\Sigma}$ such that
\begin{itemize}
  \item exactly one $x_i$ lies on each $\beta$-circle,
  \item exactly one $x_i$ lies on each $\alpha$-circle and
  \item at most one $x_i$ lies on each $\alpha$-arc.
\end{itemize}
To compute nontrivial differentials for the Floer theory, we will need to compute holomophic curves in $\overline{\Sigma} \times I_s \times \mathbb{R}_t$, where $I_s = [0,1]$ with parameter $s$ and $\mathbb{R}_t$ is $\mathbb{R}$ with parameter $t$. We will consider curves whose boundaries are on $\overline{ \boldsymbol{\alpha} } \times  \{ 1 \} \times \mathbb{R}_t$ and $\boldsymbol{\beta} \times \{ 0 \} \times \mathbb{R}_t$, asymtotic to $\mathbf{x} \times I_s$ and $\mathbf{y} \times I_s$ at $t = \pm \infty$ for $\mathbf{x}, \mathbf{y} \in \mathfrak{ S }( \mathcal{H} )$. Each of the curves carries a relative homology class in the relative homology group
\begin{eqnarray*}
& H_2 ( & \overline{\Sigma}	\times I_s \times [-\infty, +\infty], \\ 
& &( \ ( \ \overline{\boldsymbol{\alpha}} \times \{ 1 \} \cup \boldsymbol{\beta} \times \{ 0 \} \cup ((\partial \overline{\Sigma} \backslash z) \times I_s) \ ) \times [-\infty, +\infty] \ ) \cup \\
& & ( \ (\mathbf{x} \times I_s \times \{ -\infty \} ) \cup ( \mathbf{y} \times I_s \times \{ +\infty \} )   \ ) \ ).
\end{eqnarray*}
We write $\pi_2 ( \mathbf{x}, \mathbf{y} )$ as the set of these relative homology classes. \\

Note that for $B \in \pi_2 (\mathbf{x}, \mathbf{y})$, projecting $B$ onto $\overline{\Sigma}$ gives an element in $H_2 (\overline{\Sigma}, \overline{\boldsymbol{\alpha}} \cup \boldsymbol{ \beta} \cup \partial \overline{\Sigma} ) $. This is a linear combination of components of $\overline{\Sigma} \backslash ( \overline{\boldsymbol{\alpha}} \cup \boldsymbol{\beta} )$. This linear combination will be called \emph{domains}. Typically a domain is written as a linear combination of regions (connected subset of $\overline{\Sigma} \backslash (\overline{\boldsymbol{\alpha}} \cup \boldsymbol{\beta}) $.) In particular, if any $B \in \pi_2 ( \mathbf{x}, \mathbf{y} )$ is meeting $( \partial \overline{\Sigma} \backslash z ) \times I_s \times [-\infty, +\infty]$, then it can be interpreted as the corresponding domain being adjacent to the boundary of $\overline{\Sigma}$, and that gives a sequence of Reeb chords $\boldsymbol{\rho} =( \rho_1, \cdots, \rho_n) $. We call $(B, \boldsymbol{\rho} )$ a \emph{compatible pair}. \\ 

There is an operation $ * : \pi_2 ( \mathbf{x}, \mathbf{y} ) \times \pi_2 ( \mathbf{y}, \mathbf{z} ) \rightarrow \pi_2 (\mathbf{x}, \mathbf{z} )$, defined by concatenating two homology classes in the $t$ factor. In particular, if $\pi_2 ( \mathbf{x}, \mathbf{y} )$ is nonempty, then the action of $\pi_2 ( \mathbf{x}, \mathbf{x} )$ on $\pi_2 ( \mathbf{x}, \mathbf{y} )$ is free and transitive. The domain of the element in $\pi_2 ( \mathbf{x}, \mathbf{x} )$ is called \emph{periodic domain}. In addition, $\pi_2^{\partial} (\mathbf{x}, \mathbf{x} ) $ denotes a set of periodic domains not adjacent to the boundary. An element in $\pi_2^{\partial} (\mathbf{x}, \mathbf{x} ) $ is a \emph{provincial periodic domain}, and if every provincial periodic domain of a Heegaard diagram has both positive and negative coefficients, then the Heegaard diagram is called \emph{provincially admissible}.  \\

It is worth mentioning that 
\begin{itemize}
  \item if any $B \in \pi_2 ( \mathbf{x}, \mathbf{y} )$ is representing a holomorphic curve, then all the coefficients of the domain of $B$ must be nonnegative, and
  \item the operation $*$ of two classes corresponds to the sum of the respective domains.
\end{itemize}
We sometimes blur the distinction between homology classes and their domains if it does not cause confusion. \\

We define $\widehat{CFD}( \mathcal{H} )$ as the following. Let $X ( \mathcal{H} )$ be the $\mathbb{F}$-module generated by $\mathfrak{S} ( \mathcal{H})$ equipped with an action of $\mathcal{I} \subset \mathcal{A}=\mathcal{A} (- \mathcal{Z})$ (the negative sign means the algebra obtained from the pointed matched circle has an orientation opposite from the induced orientation of $\mathcal{H}$) such that for any idempotent $\iota \in \mathcal{I}$,
\begin{displaymath}
\iota \otimes \mathbf{x} := \left\{
  \begin{array}{ll}
    \mathbf{x} & \textrm{if the arc corresponding to $\iota$ is not occupied by $\mathbf{x}$} \\
    0 & \textrm{otherwise.}
  \end{array}
\right.
\end{displaymath}
Then $\widehat{CFD}( \mathcal{H} ) := \mathcal{A} \otimes_{\mathcal{I}} X ( \mathcal{H} ) $. Its differential $\delta^1$ is defined as
\begin{displaymath}
\delta^1 ( \mathbf{x}) := \sum_{\mathbf{y} \in \mathfrak{S} (\mathcal{H})} \sum_{B \in \pi_2 (\mathbf{x}, \mathbf{y}) } a^B_{\mathbf{x},\mathbf{y}} \cdot \mathbf{y},
\end{displaymath}
where 
\begin{displaymath}
a^B_{\mathbf{x}, \mathbf{y} } := \sum_{ \{ \boldsymbol{\rho} | \textrm{ind}(B, \boldsymbol{\rho})=1 \} } \sharp ( \mathcal{M}^B(\mathbf{x}, \mathbf{y} ; \boldsymbol{\rho}) ) a(- \boldsymbol{\rho} ).
\end{displaymath}
Here, $\mathcal{M}^B(\mathbf{x}, \mathbf{y} ; \boldsymbol{\rho})$ denotes the moduli space of holomorphic curves of $B$ representing the compatible pair $(B, \boldsymbol{\rho})$, and $\textrm{ind}(B, \boldsymbol{\rho})$ the expected dimension of the moduli space. In addition, for $\boldsymbol{\rho} = \{ \rho_1, \cdots, \rho_n \}$ a sequence of Reeb chords, $a (- \boldsymbol{\rho})$ be the product $a(-\rho_1) \cdots a(-\rho_n) \in \mathcal{A}$. (Again, the negative sign means that the orientation of the boundary $\partial \overline{\Sigma}$ is opposite from the induced orientation.) \\

The differential $\delta^1$ may not be well defined. In fact, there may be infinitely many homology classes in $\pi_2 ( \mathbf{x}, \mathbf{y} )$ if there is a periodic domain representing a holomorphic curve. To prevent this, we will work on a Heegaard diagram such that every periodic domain has both positive and negative coefficients. Such diagram is called \emph{admissible}, and it is shown in~\cite[Proposition 4.25]{LOT08} that every Heegaard diagram is isotopic to an admissible Heegaard diagram. (In fact, the provincial admissibility also ensures the sum is finite since the concatenation of non-provincial periodic domains of holomorphic curves produces an algebra element that equals zero.) \\

The definition of $\widehat{CFA}(\mathcal{H})$ is similar. $\widehat{CFA}( \mathcal{H} )$ is a $\mathbb{F}$-module generated by $\mathfrak{S}( \mathcal{H} )$, equipped with an action of $\mathcal{I} \subset \mathcal{A}( \mathcal{Z} ) $ such that
\begin{displaymath}
\mathbf{x} \otimes \iota := \left\{
  \begin{array}{ll}
    \mathbf{x} & \textrm{if the arc corresponding to $\iota$ is occupied by $\mathbf{x}$} \\
    0 & \textrm{otherwise.}
  \end{array}
\right.
\end{displaymath} 
$\widehat{CFA}(\mathcal{H})$ is $\mathbb{F}$-module $X(\mathcal{H})$ generated by $\mathfrak{S}$, equipped with the $\mathcal{A}_{\infty}$-module maps
\begin{displaymath}
m_{i+1} : X( \mathcal{H} ) \otimes \underbrace{ \mathcal{A}(\mathcal{Z}) \otimes \cdots \otimes \mathcal{A}( \mathcal{Z} ) }_{\textrm{$i$-times} } \rightarrow X ( \mathcal{H} )
\end{displaymath}
such that
\begin{eqnarray*}
m_{n+1}(\mathbf{x}, \rho_1, \cdots, \rho_n ) &:= &\sum_{\mathbf{y} \in \mathfrak{S}(\mathcal{H} ) } 
\sum_{ \begin{subarray}{l} B \in \pi_2 ( \mathbf{x}, \mathbf{y}) \\ \textrm{ind}(B, \boldsymbol{\rho})=1 
\end{subarray} } \sharp ( \mathcal{M}^B (\mathbf{x}, \mathbf{y}; \boldsymbol{\rho}) ) \mathbf{y}, \\
m_2(\mathbf{x},1) &:= &\mathbf{x} \\
m_{n+1} (\mathbf{x}, \cdots, 1, \cdots) &:=& 0, \quad n>1.
\end{eqnarray*} 

Although these modules are defined via a specific Heegaad diagram $\mathcal{H}$, it turns out the homotopy type of these modules are well defined. Thus, they are modules defined on bordered three-manifold (with single boundary). 

\subsection{Doubly bordered three-manifold}
The bordered three-manifold is easily extended to a three-manifold with two boundary components. A \emph{doubly bordered three-manifold} has the following data; $( Y_{12}, \Delta_1, \Delta_2, z_1, z_2, \psi_1, \psi_2, \gamma ) $. $Y_{12}$ is an oriented three-manifold with boundary $F( \mathcal{Z}_1 ) \amalg F( \mathcal{Z}_2 )$, $\Delta_i$ is a preferred disk of surface $F ( \mathcal{Z}_i )$, $z_i$ is a point on $\partial \Delta_i$, and $\psi_i$ is a parametrization of $F ( \mathcal{Z}_i )$, $i=1,2$. Moreover, $\gamma$ is an arc connecting $z_1$ and $z_2$, equipped with a framing pointing into $\Delta_i$. \\

A doubly bordered three-manifold can be realized by a Heegaard diagram with two boundaries, namely \emph{arced bordered Heegaard diagram with two boundaries}.

\begin{defn}
  An \emph{arced bordered Heegaard diagram $\mathcal{H}$ with two boundaries} is a tuple $(\Sigma, \overline{\boldsymbol{\alpha}}, \boldsymbol{\beta}, \mathbf{z})$ satisfying:
  \begin{itemize}
    \item $\overline{\Sigma}$ is a compact, genus $g$ surface with two boundary components $\partial_L \overline{\Sigma}$ and $\partial_R \overline{\Sigma}$.
    \item $\boldsymbol{\beta}$ is $g$-tuple of pairwise disjoint curves in the interior of $\overline{\Sigma}$.
    \item $\overline{\boldsymbol{\alpha}} = \{ \boldsymbol{\alpha}^{a,L}=\{ \alpha_1^{a,L}, \cdots, \alpha_{2l}^{a,L} \}, \ \boldsymbol{\alpha}^{a,R}=\{\alpha_1^{a,R}, \cdots, \alpha_{2r}^{a,R} \}, \ \boldsymbol{\alpha}^{c}=\{\alpha_1^c, \cdots, \alpha_{g-l-r}^c \} \}$, is a collection of pairwise disjoint embedded arcs with  boundary on $\partial_L \overline{\Sigma}$ (the $\alpha_i^{a,L}$), arcs with boundary on $\partial_R \overline{\Sigma}$ (the $\alpha_i^{a,R}$), and circles (the $\alpha_i^c$) in the interior of $\overline{\Sigma}$.
    \item $\mathbf{z}$ is a path in $\overline{\Sigma} \backslash (\overline{\boldsymbol{\alpha}} \cup   \boldsymbol{\beta})$ between $\partial_L \overline{\Sigma}$ and $\partial_R \overline{\Sigma}$,
  \end{itemize}
such that $\overline{ \boldsymbol{\alpha} }$ intersects $\boldsymbol{\beta}$ transversely, and $\overline{\Sigma} \backslash \overline{\boldsymbol{\alpha}}$ and $\overline{\Sigma} \backslash \boldsymbol{\beta}$ are connected.
\end{defn}

Note that an arced bordered Heegaard diagram $\mathcal{H}$ specifies two pointed matched circles on its ``left'' and ``right'' boundary. (The choice of ``left'' and ``right'' are arbitrary.) These are 
\begin{eqnarray*}
\mathcal{Z}_L ( \mathcal{H} ) & = & ( \partial_L \overline{\Sigma}, \boldsymbol{\alpha}^{a,L} \cap \partial_L \overline{\Sigma}, M_L, \mathbf{z} \cap \partial_L \overline{\Sigma}), \quad \textrm{and} \\
\mathcal{Z}_R ( \mathcal{H} ) & = & ( \partial_R \overline{\Sigma}, \boldsymbol{\alpha}^{a,R} \cap \partial_R \overline{\Sigma}, M_R, \mathbf{z} \cap \partial_R \overline{\Sigma}).
\end{eqnarray*}

The construction of a doubly bordered three-manifold is similar to the construction of a single boundary case. For an arced bordered Heegaard diagram $\mathcal{H}$, cut open the diagram along the arc $\mathbf{z}$. The resulting diagram is a bordered Heegaard diagram with a single boundary, which will be written as $\mathcal{H}_{dr}$. Then, construct a bordered three-manifold $Y( \mathcal{H}_{dr} )$. The boundary of $Y( \mathcal{H}_{dr} )$ is a surface that can be decomposed as a connected sum $F( \mathcal{Z}_L ) \sharp F( \mathcal{Z}_R )$. Finally, attach a three-dimensional two-handle along the connect sum annulus. \\ 

The three-manifold $Y( \mathcal{H} ) := Y ( \mathcal{H}_{dr} ) \cup \{ \textrm{two-handle} \}$ has the following properties.
\begin{itemize}
  \item It has two boundary surfaces $F(\mathcal{Z}_L)$ and $F(\mathcal{Z}_R)$ with parametrization given by $\boldsymbol{\alpha}^{a,L}$ and $\boldsymbol{\alpha}^{a,R}$, respectively.
  \item Each boundary surface has a preferred disk bounded by $\partial_L \overline{\Sigma}$ or $\partial_R \overline{\Sigma}$.
  \item The cut open of the Heegaard diagram $\mathcal{H}$ would result in two arcs $\mathbf{z}^+$ and $\mathbf{z}^-$ on the deleted neighborhood of $\mathbf{z}$. Then, the arc $\mathbf{z}^+$, thought as a subset of the boundary of $Y( \mathcal{H}_{dr} )$, is the framed arc in $Y( \mathcal{H} )$ connecting $z_1$ and $z_2$.
\end{itemize}

For an arced Heegaard diagram $\mathcal{H}$, the type-$DD$ bimodule $\widehat{CFDD} ( \mathcal{H} )$ is defined almost the same as in $\widehat{CFD}$. $\widehat{CFDD} ( \mathcal{H} )$ is a left-left $\mathbb{F}$-$\mathbb{F}$-module generated by $\mathfrak{S} ( \mathcal{H}_{dr} )$, equipped with two left actions of $\mathcal{I}_L \subset \mathcal{A}_L := \mathcal{A} (-\mathcal{Z}_L) $ and $\mathcal{I}_R \subset \mathcal{A}_R := \mathcal{A}(-\mathcal{Z}_R) $ such that for $\iota_L \in \mathcal{I}_L$ and $\iota_R \in \mathcal{I}_R$, 
\begin{displaymath}
\iota_L \otimes \iota_R \otimes \mathbf{x} := \left\{
  \begin{array}{ll}
    \mathbf{x} & \textrm{if the arc corresponding to $\iota_L$ and $\iota_R$ are not occupied by $\mathbf{x}$} \\
    0 & \textrm{otherwise.}
  \end{array}
\right.
\end{displaymath} 

Then $\widehat{CFDD}(\mathcal{H} ) = \mathcal{A}_L \otimes \mathcal{A}_R \otimes \mathfrak{S}( \mathcal{H}_{dr} )$ with the differential
\begin{displaymath}
\delta^1 ( \mathbf{x}) := \sum_{\mathbf{y} \in \mathfrak{S} (\mathcal{H}_{dr})} \sum_{B \in \pi_2 (\mathbf{x}, \mathbf{y}) } a^B_{\mathbf{x},\mathbf{y}} \cdot \mathbf{y},
\end{displaymath}
where 
\begin{displaymath}
a^B_{\mathbf{x}, \mathbf{y} } := \sum_{ \{ \boldsymbol{\rho}_L, \boldsymbol{\rho}_R | \textrm{ind}(B, \boldsymbol{\rho}_L, \boldsymbol{\rho}_R )=1 \} } \sharp ( \mathcal{M}^B(\mathbf{x}, \mathbf{y} ; \boldsymbol{\rho}_L, \boldsymbol{\rho}_R) ) a(- \boldsymbol{\rho}_L ) \otimes a(- \boldsymbol{\rho}_R ).
\end{displaymath}

Similarly, a type-$AA$ bimodule $\widehat{CFAA}( \mathcal{H} )$ is defined by a right-right $\mathbb{F}$-$\mathbb{F}$ bimodule generated by $\mathfrak{S} ( \mathcal{H}_{dr} )$ with right-right actions of idempotents.
\begin{displaymath}
\mathbf{x} \otimes \iota_L \otimes \iota_R := \left\{
  \begin{array}{ll}
    \mathbf{x} & \textrm{if the arc corresponding to $\iota_L$ and $\iota_R$ are occupied by $\mathbf{x}$} \\
    0 & \textrm{otherwise.}
  \end{array}
\right.
\end{displaymath} 
The type-$AA$ module maps are 
\begin{displaymath}
m_{n+m+1}(\mathbf{x}, \rho^L_1, \cdots, \rho^L_n, \rho^R_1, \cdots, \rho^R_m ) := \sum_{\mathbf{y} \in \mathfrak{S}(\mathcal{H} ) } 
\sum_{ \begin{subarray}{l} B \in \pi_2 ( \mathbf{x}, \mathbf{y}) \\ \textrm{ind}(B, \boldsymbol{\rho}^L, \boldsymbol{\rho}^R )=1 
\end{subarray} } \sharp ( \mathcal{M}^B (\mathbf{x}, \mathbf{y}; \boldsymbol{\rho}^L, \boldsymbol{\rho}^R) ) \mathbf{y}.
\end{displaymath} 

Lastly, the expected dimension of the moduli space of $\mathcal{M}^B (\mathbf{x}, \mathbf{y} ; \overrightarrow{\rho^L}, \overrightarrow{\rho^R})$, or $\textrm{ind}(B, \overrightarrow{\rho^L}, \overrightarrow{\rho^R})$ is computed by the formula below.
\begin{displaymath}
  \textrm{ind}(B, \boldsymbol{\rho}) = e(B) + n_{\mathbf{x}}(B) +
  n_{\mathbf{y}}(B) + |\overrightarrow{\rho^L}| +
  |\overrightarrow{\rho^R}| + \iota({\overrightarrow{\rho^L}}) +
  \iota({\overrightarrow{\rho^R}}),
\end{displaymath}
where $e(B)$ is \emph{Euler measure}, $n_{\mathbf{x}}(B)$ sum of average of local multiplicities surrounding generator $\mathbf{x}$, $|\overrightarrow{\rho^L}|$ number of Reeb chords in the sequence
$\overrightarrow{\rho^L}$, and $\iota( \overrightarrow{\rho^L} )$ \emph{linking number} of sequence $\overrightarrow{\rho^L}$. See \cite[Definition 5.11]{LOT08}. 

\subsection{Pairing Theorem}
The type-$A$ module and type-$D$ modules can be paired, which results in the classical Heegaard Floer homology of a closed three-manifold. The original pairing theorem is given in \cite[Theorem 1.3]{LOT08}. For any two three-manifolds $Y_1$ and $Y_2$ with $\partial Y_1 = F (\mathcal{Z}) = - \partial Y_2$,
\begin{displaymath}
\widehat{CFA}( Y_1) \widetilde{\otimes}_{\mathcal{A}(\mathcal{Z})} \widehat{CFD}( Y_2) \cong \widehat{CF} (Y_1 \cup_{F( \mathcal{Z} )} Y_2 )
\end{displaymath} 
where $\widetilde{\otimes}$ denotes the derived tensor product. The bimodule version of the pairing theorem is also given in~\cite{LOT11}. If $Y_{12}$ is a doubly bordered three-manifold with boundary $F( \mathcal{Z}_1 ) \amalg F( \mathcal{Z}_2) $ and $Y_1$ is a bordered three-manifold with boundary $F( \mathcal{Z}_1 )$, then 
\begin{displaymath}
\widehat{CFD}(Y_1 \cup_{F(\mathcal{Z}_1) } Y_{12} ) \cong \widehat{CFA}(Y_1) \widetilde{\otimes}_{\mathcal{A} (\mathcal{Z}_1) } \widehat{CFDD}(Y_{12}). 
\end{displaymath}
There exists many other variations of the pairing theorem. Interested readers should refer to~\cite{LOT11}.

\section{Computation of the bordered Floer bimodule of the (2,2n)-torus link}
\label{sec:main}

\subsection{Schubert normal form and diagram of 2-bridge link complements}

As we will mainly focus on 2-bridge links, it is useful to mention \emph{Schubert normal form} of a 2-bridge links (or knots). Let $p$ be an even positive integer and $q$ be an integer such that $0<q<p$ and $\mathrm{gcd}(p,q)=1$. Let us consider a circle with $2p$ marked point on its boundary. Choose a point and label it $a_0$. Label the other points $a_1, \cdots, a_{2p-1}$ in a clockwise direction. Then, connect $a_i$ and $a_{2p-i}$ with a straight line, $i=1, \cdots, p-1$. Finally, connect $a_0$ and $a_p$ with an
underbridge, a straight line that crosses below all of the other straight lines. \\

Now consider two copies of such circle. Draw arcs between these two circles so that each arc is connecting $a_i$ on the left circle to $a_{q-i}$ on the right circle (the labeling is modulo $2p$). These arcs
should not intersect any of the straight lines and arcs. The resulting diagram gives a link that we denote $S(p,q)$. The diagram is called \emph{Schubert normal form} of the link. See Figure~\ref{fig:83link}. More detailed description, especially about the Schubert normal form of 2-bridge knot can be found in~\cite[Chapter 2]{Ras02}. \\

\begin{figure}
  \centering
  \includegraphics[width=1\textwidth]{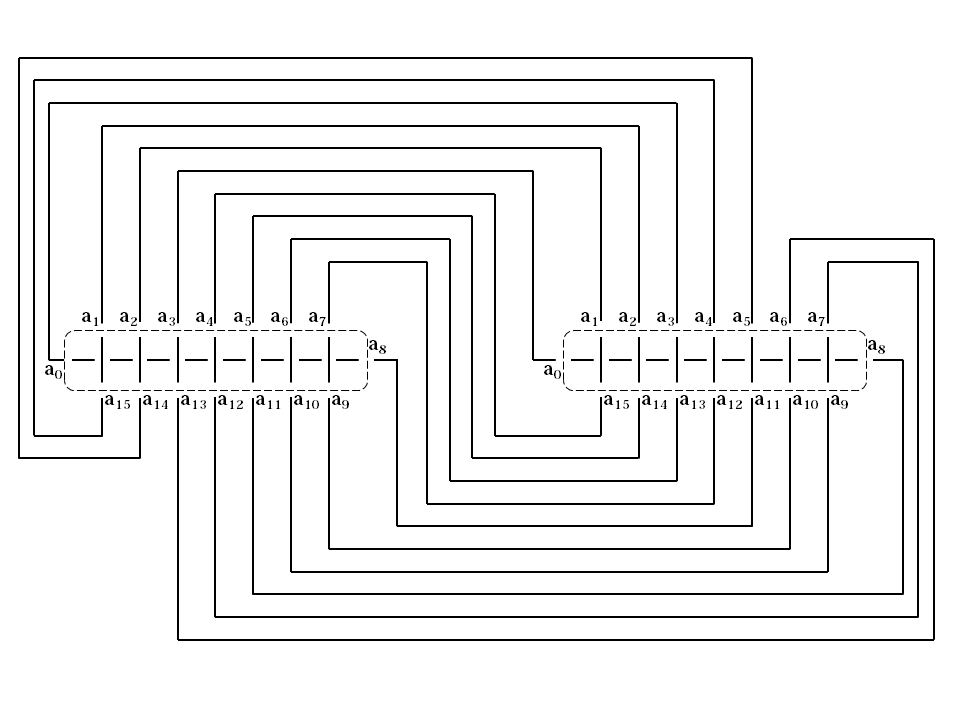}
  \caption{Schubert normal form of $S(8,3)$-link. According to Thistlethwaite's table, it is $L5a1$ link.}
  \label{fig:83link}
\end{figure}

Recall that 2-bridge link $L$ is a link in $S^3$ that admits a link diagram with two maxima and two minima. Let $B_1$ and $B_2$ be small neighborhoods of those two maxima. Consider $(S^3 \backslash \nu L) \backslash (B_1 \cup B_2)$. Drilling a tunnel connecting $B_1$ and $B_2$ gives a three-manifold $Y$ with single boundary, and the boundary is a genus 2 surface. Also, the longitudes $\lambda_L$ and $\lambda_R$ of the left and right components of $L$ are considered as curves on $\partial (\nu L)$; therefore the longitudes are also curves on the boundary of the drilled three-manifold. \\

The resulting manifold can be viewed as a handlebody with one zero-handle and two one-handles attached to it. To get a bordered Heegaard diagram, we will apply the following procedures on the boundary of the three-manifold. First, apply an isotopy of the boundary surface so that the longitudes have the Schubert normal form. Then, draw two circles $\beta_1$ and $\beta_2$ on the boundary surface so that they are parallel to the core of the one-handles on the boundary of the one-handles. Next, draw the meridians $\mu_L$ and $\mu_R$ on the belt sphere of each one-handle. Finally, make two punctures at the two intersections of meridians and longitudes and relabel $\lambda_L$ to $\alpha_1^{a,L}$ and $\mu_L$ to $\alpha_2^{a,L}$ (respectively, $\lambda_R$ to $\alpha_1^{a,R}$ and $\mu_R$ to $\alpha_2^{a,R}$). \\

In particular, if $L$ equals the $(2,2n)$-torus link, then we can draw an arc on the surface connecting two punctures so that the arc is not intersecting $\overline{\boldsymbol{\alpha}}$ or $\boldsymbol{\beta}$ curves. From now on, the punctured boundary surface equipped with $\overline{\boldsymbol{\alpha}}$ or $\boldsymbol{\beta}$ curves and the arc is written as $\mathcal{H}$, which is the arced bordered Heegaard diagram of $(2,2n)$-torus link complement. 
\begin{figure}
  \centering
  \includegraphics[width=1.0\textwidth]{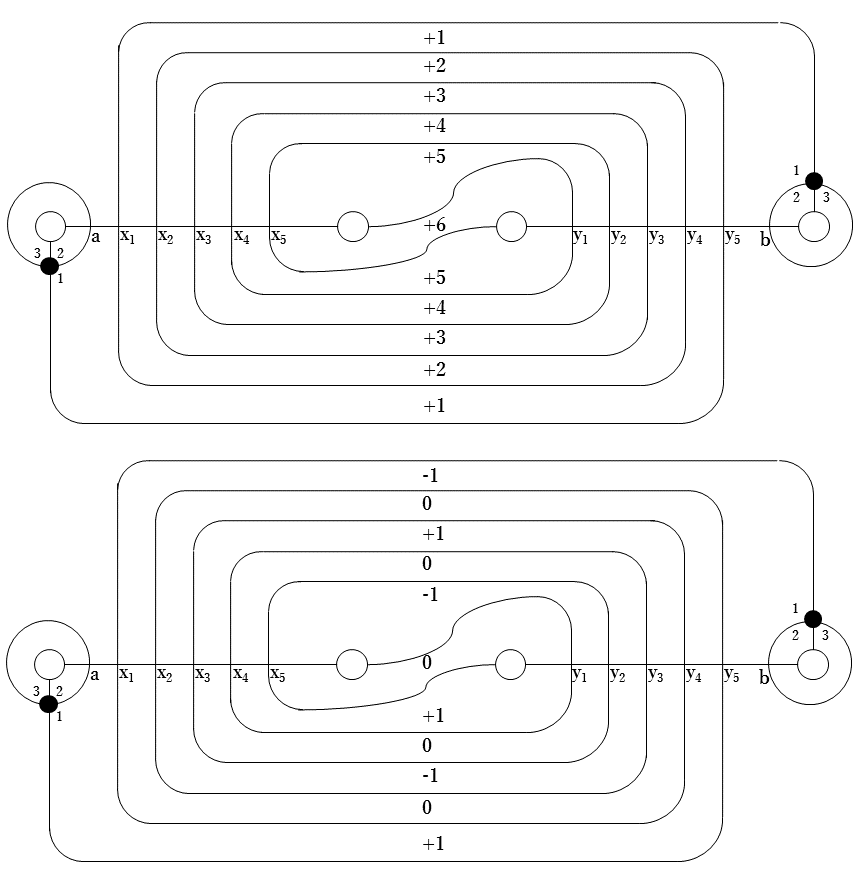}
  \caption{These two diagrams represents two periodic domains of the bordered Heegaard diagram of (2,6)-torus link, where the black dots represent left and right punctures.}
  \label{fig:periodicdomain}
\end{figure}

\begin{figure}
  \centering
  \includegraphics[width=1.1\textwidth]{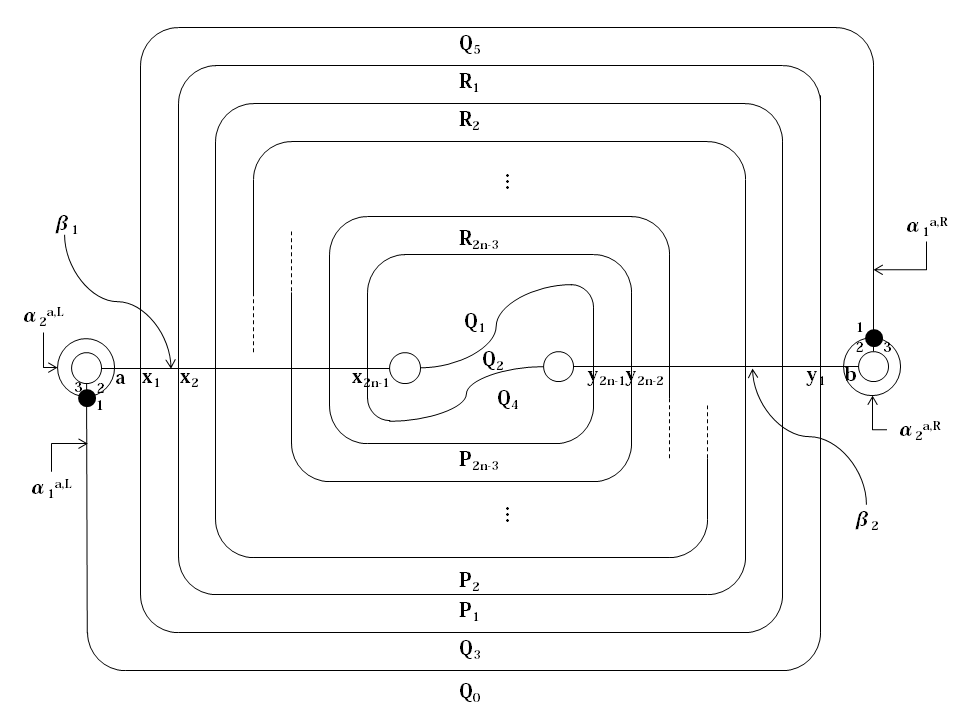}
  \caption{A general diagram of $(2,2n)$ torus link. Domain $Q_0$ has a framed arc. The orientation on the boundaries is opposite from the usual ``right hand'' orientation.}
  \label{fig:generaldiagram}
\end{figure}

\begin{rem}
Readers should be aware that connecting the left and right punctures with an (framed) arc is not always possible. In fact, a domain that is adjacent to both punctures does not exist except for the $(2,2n)$-torus link case. To fix this, choose $\mu_L$ or $\mu_R$ and apply a finger move on the chosen meridian along the longitude so that the resulting puncture is on the domain that is adjacent to the other puncture.
\end{rem}

\subsection{Computation of the type-DD module differential}

Now, we will compute $\widehat{CFDD}(\mathcal{H})$, where $\mathcal{H}$ is the Heegaard diagram of the $(2,2n)$-torus link complement constructed in the previous section. The Heegaard diagram is given in Figure~\ref{fig:generaldiagram}. \\

First, we will see whether the diagram $\mathcal{H}$ is provincially admissible. Second, we will investigate the genus-zero rectangular domains that cause a nontrivial differential. Then, using the result as a building block, we will consider domains of higher genus and the moduli space of homolorphic curves of the domains. The differentials associated to the higher genus domains are computed by $\mathcal{A}_{\infty}$-relations, dualizing $\widehat{CFDD}$-bimodule to $\widehat{CFAA}$-bimodule. \\

\textbf{Periodic domain} First, we investigate periodic domains $\pi_2 (\mathbf{x},\mathbf{x})$. It is well known that $\pi_2 (\mathbf{x},\mathbf{x}) \cong H_2(Y(\mathcal{H}), \partial Y (\mathcal{H})) \cong \mathbb{Z} \oplus \mathbb{Z}$ by the Mayer-Vietoris sequence. Thus, there are two linearly independent periodic domains in the diagram. The proof can be found in \cite[Lemma 2.6.1]{Lip06}, or \cite[Lemma 4.18]{LOT08}. In their proof, they use the isomorphism
\begin{displaymath}
  \boldsymbol{\pi}_2 (\mathbf{x}, \mathbf{x}) \cong H_2 ( \Sigma' \times [0,1] , (\overline{\boldsymbol{\alpha}} \times \{ 1 \}) \cup ( \boldsymbol{\beta} \times \{ 0 \} ) )
\end{displaymath}
where $\Sigma' = (\overline{\Sigma} / \partial \overline{\Sigma} ) \backslash \{ z \}$. The isomorphism given above is proved by investigating the long exact sequence of pair $(\Sigma' \times [0,1] , (\overline{\boldsymbol{\alpha}} \times \{ 1 \}) \cup ( \boldsymbol{ \beta } \times \{ 0 \} ) )$. That is,
\begin{eqnarray*}
  \cdots & \rightarrow & \underbrace{H_2 (\Sigma' \times [0,1])}_{\cong
  0} \rightarrow H_2 ( \Sigma'   \times [0,1] ,
  (\overline{\boldsymbol{\alpha}} \times \{ 1 \})
  \cup ( \boldsymbol{\beta} \times \{ 0 \} ) ) \\
  & \rightarrow & H_1 ( ( \overline{ \boldsymbol{\alpha} } \times \{
  1 \} ) \cup ( \boldsymbol{ \beta } \times \{ 0 \} ) ) \rightarrow
  H_1 ( \Sigma' ).
\end{eqnarray*}
Thus, the periodic domain $\boldsymbol{\pi}_2 (\mathbf{x}, \mathbf{x}) \cong \textrm{ker}( H_1 ( \overline{ \boldsymbol{\alpha} } / \partial \overline{ \boldsymbol{\alpha} } ) \oplus H_1 ( \boldsymbol{ \beta } ) \rightarrow H_1 ( \overline{\Sigma} / \partial \overline{ \Sigma } ) )$. This isomorphism enables us to find periodic domains from a bordered Heegaard diagram by choosing combinations of $\overline{ \boldsymbol{\alpha} }$ and $\boldsymbol{ \beta }$ curves such that the sum of their image in $H_1 ( \overline{\Sigma} / \partial \overline{\Sigma} )$ equals zero. We briefly describe how to find the periodic domain from such combinations. Explicitly, first choose any orientation on the longitude $\alpha_1^{a,L}$ ($\alpha_1^{a,R}$, respectively). This induces the orientation of $\beta_1$ ($\beta_2$, respectively) as follows. For example, if the orientation of $\alpha_1^{a,L}$ is in a counterclockwise direction, then the orientation of $\beta_1$ is from right to left in the diagram. Then, we impose the coefficient zero to the outermost region that contains the framed arc. Starting from the outermost region, we give coefficients to regions adjacent to it according to the following rule. Suppose we have two adjacent regions $A$ and $B$ such that the coefficient of $A$ equals $l$ and the coefficient of $B$ is not determined. If we can reach region $B$ from region $A$ by crossing a curve of multiplicity $k$ from right to left (notion of ``left'' and ``right'' is justified since we have orientation of curves), we give the region $B$ coefficient $k+l$; otherwise we give coefficient $-k+l$. If we can give coefficients to all regions consistently in this way, then the orientations given to curves $\overline{\boldsymbol{\alpha}}$ and $\boldsymbol{\beta}$ is boundary in $H_1( \overline{\Sigma} / \partial \overline{\Sigma} )$. \\

Since there are two possible choices of orientations of longitudes up to sign, we find two generators of $\pi_2 (\mathbf{x},\mathbf{x})$. Then the periodic domains are

\begin{displaymath}
  Q_3 + Q_5 + \sum_{i=1}^{2n-3} (i+1)(P_i + R_i) + (n+2)(Q_1+Q_4) +
  (n+3)Q_2
\end{displaymath}

and

\begin{displaymath}
  Q_3 -Q_5 + \sum_{i=1}^{2n-3} \frac{(1+(-1)^i)}{2}(P_i - R_i) + Q_4 -
  Q_1.
\end{displaymath}

See also Figure~\ref{fig:periodicdomain}. \\

Thus, this diagram is provincially admissible; in fact, there is no provincial periodic domain here. \\

\textbf{Generators} According to the labeling given in the diagram, there are $2n^2+2n$ generators which are classified into 4 groups.

\begin{displaymath}
  \left\{ \begin{array}{llll}
    \mathbf{x_i y_j} & \textrm{where $i$ and $j$ have same parity} \\
    \mathbf{a y_i} & \textrm{where $i$ is even} \\
    \mathbf{x_i b} & \textrm{where $i$ is even} \\
    \mathbf{a y_i}, \mathbf{x_j b} & \textrm{where $i$ and $j$ are odd}\\
  \end{array} \right.
\end{displaymath}

From now on, we will disregard generators of the last group because of the following reason. The main purpose of the bordered Floer homology is to compute the Heegaard Floer homology of a three-manifold obtained by gluing along the boundaries of two three-manifolds with homeomorphic boundaries. In the link complement case, we glue the link complement and solid tori. Typically, a bordered Heegaard diagram of a solid torus is a genus one surface with a puncture, equipped with $\boldsymbol{\beta} = \{ \beta_1 \}$ and $\overline{ \boldsymbol {\alpha} } = \{ \alpha_1^a , \alpha_2^a \}$. In particular, these $\alpha_i^a$ arcs are glued to $\alpha_j^{a,L}$ or $\alpha_i^{a,R}$ of the doubly bordered diagram of the link complement, and every generator of the diagram of the solid torus is occupying exactly one $\alpha$ arc. Therefore, after pairing two diagrams of the solid tori to both sides of the diagram of the link complement, the generators of the last kind cannot appear in the generator set of the resulting diagram.\\

\begin{rem}
In \cite[Chapter 3]{LOT08}, they have decomposed strands algebra $\mathcal{A}(\mathcal{Z})$ into the direct sum of $\mathcal{A}(\mathcal{Z},i)$, $i \in \{ -1, 0, 1 \}$, based on the number of Reeb chords $k+i$, where $k$ is the genus of the boundary surface. Likewise, we can decompose $\widehat{CFDD} (\mathcal{H})$ as follows so that the idempotent acts nontrivially on respective summands. 
\begin{displaymath}
  \widehat{CFDD} (\mathcal{H}) = \bigoplus_{i=-1}^{1} \widehat{CFDD}
  (\mathcal{H}, i)
\end{displaymath}
where
\begin{itemize}
  \item $\widehat{CFDD} (\mathcal{H}, -1)$ consists of a generator that occupies $\alpha_1^{a,R}$ and $\alpha_2^{a,R}$;
  \item $\widehat{CFDD} (\mathcal{H}, +1)$ consists of a generator that occupies $\alpha_1^{a,L}$ and $\alpha_2^{a,L}$;
  \item $\widehat{CFDD} (\mathcal{H}, 0)$ consists of all other generators.
\end{itemize}
The first three groups of the generators belong to $\widehat{CFDD} ( \mathcal{H}, 0 )$, but the last group of generators does not. \\

Clearly, only the generators in summand $\widehat{CFDD} (\mathcal{H}, 0)$ have a nontrivial contribution to any tensor product with $\widehat{CFA}$ module, considering the only nontrivial algebra of $\widehat{CFA}$ and $\widehat{CFD}$ is $\mathcal{A} (\mathcal{Z}, 0)$. Moreover, since $\mathcal{A} (\mathcal{Z}, -1)$ and $\mathcal{A} (\mathcal{Z}, +1)$ are quasi-isomorphic to $\mathbb{F}_2$ \cite[Example 3.25]{LOT08}, any invertible bimodule over either one of these algebras is also quasi-isomorphic to $\mathbb{F}_2$,~\cite[Chapter 10]{LOT11}.
\end{rem}

From now on, we will be only focusing on the generators belonging to $\widehat{CFDD} (\mathcal{H}, 0)$. \\

The differential $\delta^1 : \mathfrak{S}(\mathcal{H}) \rightarrow \mathcal{A}(- \partial_L \overline{\Sigma}) \otimes \mathcal{A}(- \partial_R \overline{\Sigma}) \otimes \mathfrak{S}(\mathcal{H})$  maps a generator $\mathbf{x} \in \mathfrak{S}(\mathcal{H})$ to $\sum \rho_I \otimes \sigma_J \otimes \mathbf{y}$, where $I,J \in \{ \phi, 1, 2, 3, 12, 23, 123\}$. Here, $\rho_I$ means an algebra element that comes from the left boundary strands algebra and $\sigma_J$, the right strands algebra. To investigate $\delta^1$ actions on generators, it is convenient to classify the resulting terms by their strands algebra elements. \\

\textbf{Algebra element 1} We should find all provincial domains. We claim that only rectangular domains contribute to the differential $\delta^1$.

\begin{lem}
  Every non-rectangular domain with \emph{ind}$(B, \boldsymbol{\rho})=1$, its sequence of Reeb chords $\boldsymbol{\rho}$ is nonempty.
\end{lem}

\begin{proof}
Suppose there is a non-rectangular provincial domain (in this case, an annulus) that has a nontrivial contribution to differential $\delta^1$. Then, the number of the corners of the domain must be two. This claim is justified by considering the number of corners of different types. Since the number of corners of any domain should not exceed four, there are only 5 possibilities;
\begin{itemize}
  \item four $270^{\circ}$ corners
  \item four $90^{\circ}$ corners
  \item three $270^{\circ}$ corners and one $90^{\circ}$ corner
  \item one $270^{\circ}$ corner and three $90^{\circ}$ corners
  \item two $270^{\circ}$ corners and two $90^{\circ}$ corners.
\end{itemize}
Since the domain was assumed to be provincial, it must be a combination of the regions $P_1, \cdots, P_{2n-3}$ and $R_1, \cdots R_{2n-3}$. Considering the index formula $e(A) + n_{\mathbf{x}}(A) + n_{\mathbf{y}}(A)$, the indices of the first three cases cannot be one. Likewise, we can easily rule out the last case. The fourth case does not exist due to the following reason; since the shape of the domain is an annulus, the $270^{\circ}$ corner must be on the boundary of the domain. Then, the other boundary must have two $90^{\circ}$ corners. If not, i.e, if one boundary component has all three $90^{\circ}$ corners, then there cannot be a holomorphic involution interchanging inner and outer boundaries. See~\cite[Lemma 9.4]{OZ02}. Thus, one boundary has two $90^{\circ}$ corners and the other boundary has one $90^{\circ}$ corner and one $270^{\circ}$ corner. In particular, the boundary that has two $90^{\circ}$ corners should consist of one $\overline{ \boldsymbol{ \alpha } }$ curve and one $\boldsymbol{ \beta }$ curve, and the intersections have to be $90^{\circ}$. However, such a boundary cannot be obtained by any combination of the domains in Figure~\ref{fig:generaldiagram}.
\end{proof}

Therefore, $P_1, \cdots, P_{2n-3}$ and $R_1, \cdots R_{2n-3}$ are the only regions not adjacent to the boundaries, so rectangular domains obtained by these regions are the only provincial domains. Such combinations of the regions are written explicitly as shown below.

\begin{eqnarray*}
  & & P_i, \quad P_i + R_{i+1} + P_{i+2}, \cdots \\
  & & P_i + P_{i+1} + P_{i+2}, \quad P_i + \cdots + P_{i+4},\cdots, \quad P_1 + \cdots P_{2n-3}, \\
  & & R_i, \quad R_i + P_{i+1} + R_{i+2}, \cdots \\
  & & R_i + R_{i+1}, R_{i+2}, \cdots, \quad R_1 + \cdots + R_{2n-3}\\
\end{eqnarray*}

All of these domains are rectangular, so each of these domains contribute to the nontrivial differential with algebra element $1$. In terms of generators,

\begin{displaymath}
  \mathbf{x_i y_j} \mapsto
  \left\{
  \begin{array}{lllll}
    \mathbf{x_{j-1} y_{i+1}} + \mathbf{x_{i+1} y_{j-1}} \quad \textrm{if $j - i > 2$}
    \\[0.3pc]
    \mathbf{x_{j+1} y_{i-1}} + \mathbf{x_{i-1} y_{j+1}} \quad \textrm{if $i - j > 2$} \\[0.3pc]
    \mathbf{x_{i+1} y_{j-1}} \quad \textrm{if $j - i = 2$} \\[0.3pc]
    \mathbf{x_{i-1} y_{j+1}} \quad \textrm{if $i - j = 2$} \\[0.3pc]
    0 \quad \textrm{if $i=j$} \\[0.3pc]
  \end{array}
  \right.
\end{displaymath}
\\

\textbf{Algebra element $\rho_1$ and $\sigma_1$.} First, consider the algebra element $\rho_1$. Domain $Q_3$ is adjacent to the Reeb chords of algebra element $\rho_1$. Note that if the multiplicity of the domain $Q_3$ is greater than 1, then it cannot contribute to the nontrivial differential. (If so, then it will produce the algebra element $\rho_1 \cdot \rho_1$, which equals zero.) We list the possible domains that result in nontrivial differentials, as written below.
\begin{eqnarray*}
\quad & & Q_3, \quad Q_3 + P_1 + P_2, \quad Q_3 + P_1 + P_2 + P_3 + P_4, \cdots, \\
\quad & & Q_3 + R_1 + P_2, \quad Q_3 + R_1 + P_2 + R_3 + P_4 \cdots.
\end{eqnarray*}
All such domains are rectangular domains containing $Q_3$. These domains are all quadrilateral, and the dimension and the modulo two count of the moduli spaces are obvious. The differentials obtained from these domains are listed below.
\begin{displaymath}
  \mathbf{a y_{2k}} \mapsto
  \left\{ \begin{array}{ll}
  \rho_1 \otimes (\mathbf{x_1 y_{2k-1}} + \mathbf{x_{2k-1} y_1}) \quad \textrm{if
  $k \neq 1$} \\[0.3pc]
  \rho_1 \otimes \mathbf{x_1 y_1} \quad \textrm{otherwise.}
  \end{array}
  \right.
\end{displaymath}

Differentials involving $\sigma_1$ can be found in a parallel manner, by using the symmetry of the diagram.

\begin{displaymath}
  \mathbf{x_{2k} b} \mapsto
  \left\{ \begin{array}{ll}
  \sigma_1 \otimes (\mathbf{x_{2k-1} y_1} + \mathbf{x_1 y_{2k-1}}) \quad \textrm{if
  $k \neq 1$} \\[0.3pc]
  \sigma_1 \otimes \mathbf{x_1 y_1} \quad \textrm{otherwise.}
  \end{array}
  \right.
\end{displaymath}
\\

\textbf{Algebra element $\rho_3$ and $\sigma_3$.} Similarly, domains adjacent to $\rho_3$ are all listed
\begin{displaymath}
  Q_1, \quad Q_1 + R_{2n-3} + R_{2n-2}, \quad Q_1 + R_{2n-3} +
  R_{2n-4} + R_{2n-5} + R_{2n-6}, \cdots
\end{displaymath}
and,
\begin{displaymath}
  Q_1 + P_{2n-3} + R_{2n-4}, \quad Q_1 + P_{2n-3} + R_{2n-4} +
  P_{2n-5} + R_{2n-6}, \cdots
\end{displaymath}

Domains adjacent to $\sigma_3$ are similar. We get the differentials below.

\begin{displaymath}
  \mathbf{a y_{2k}} \mapsto
  \left\{ \begin{array}{ll}
  \rho_3 \otimes (\mathbf{x_{2k+1} y_{2n-1}} + \mathbf{x_{2n-1} y_{2k+1}}) \quad \textrm{if
  $k \neq n-1$} \\
  \rho_3 \otimes \mathbf{x_{2n-1} y_{2n-1}} \quad \textrm{otherwise.}
  \end{array}
  \right.
\end{displaymath}

\begin{displaymath}
  \mathbf{x_{2k} b} \mapsto
  \left\{ \begin{array}{ll}
  \sigma_3 \otimes (\mathbf{x_{2n-1} y_{2k+1}} + \mathbf{x_{2k+1} y_{2n-1}}) \quad \textrm{if
  $k \neq n-1$} \\
  \sigma_3 \otimes \mathbf{x_{2n-1} y_{2n-1}} \quad \textrm{otherwise.}
  \end{array}
  \right.
\end{displaymath}
\\

\textbf{Algebra element $\rho_2 \otimes \sigma_2$.} The domain $Q_2$ adjacent to $\rho_2$ is adjacent to $\sigma_2$ as well. So, this is the one and only domain where the algebra element $\rho_2 \otimes \sigma_2$ occurs. Thus, we have $\mathbf{x_{2n-1} y_{2n-1}} \mapsto \rho_2 \otimes \sigma_2 \otimes \mathbf{ab}$. \\

\textbf{Algebra element $\rho_3 \otimes \sigma_1$ and $\rho_1 \otimes \sigma_3$.} There are two domains which contribute to $\rho_3 \otimes \sigma_1$; those are $Q_1 + R_1 + R_2 + \cdots R_{2n-3} + Q_5$ and $Q_1 + P_1 + R_2 + P_3 + R_4 + \cdots R_{2n-4} + P_{2n-3} + Q_5$. This gives $\mathbf{ab} \mapsto \rho_3 \otimes \sigma_1 \otimes (\mathbf{x_1 y_{2n-1}} + \mathbf{x_{2n-1} y_1})$. Again, using the symmetry of the diagram, $\mathbf{ab} \mapsto \rho_1 \otimes \sigma_3 \otimes (\mathbf{x_1 y_{2n-1}} + \mathbf{x_{2n-1} y_1})$. \\

Now, we will work on differentials whose domains are non-rectangular. To find holomorphic curves of such domains, we will dualize $\widehat{CFDD}$ to $\widehat{CFAA}$ so that we can use the $\mathcal{A}_{\infty}$ structure of it and ensure the existence of holomorphic curves and their count (modulo two). \\

\textbf{Algebra element containing $\rho_{12}$.} To take advantage of the $\mathcal{A}_{\infty}$-structure of $\widehat{CFAA}$, the orientation of two boundaries of the Heegaard diagram has to be reversed. We let $\overline{\rho}_I$ denote (respectively, $\overline{\sigma}_I$) the algebra element of the left strands algebra $\mathcal{A}(\mathcal{Z})$ (respectively, the right strands algebra); that is, an orientation reversing diffeomorphism $R : - S^1 \backslash \{z\} \rightarrow S^1 \backslash \{z\}$ induces a map $R_* : \mathcal{A} ( - \mathcal{Z} ) \rightarrow \mathcal{A} ( \mathcal{Z} )$ that maps $R_*(\rho_1) = \overline{\rho}_3$, $R_*(\rho_2) = \overline{\rho}_2$, $R_*(\rho_3) = \overline{\rho}_1$, and so on. The right boundary is similar.\\

Returning to $\widehat{CFDD}$, the domains contributing to $\rho_{12}$ must contain $Q_2$ and $Q_3$. Clearly $Q_2 + Q_3$ has more than four corners, so we will consider $Q_2 + Q_3 + Q_4$ instead to get the domain of four corners. This domain possibly contributes to the differential from $\mathbf{x_{2n-2} y_2}$ to $\mathbf{x_1 y_1}$. The only possible Maslov index one interpretation is $\mathcal{M} ( \mathbf{x_{2n-2} y_2}, \mathbf{x_1 y_1} ; \overline{\rho}_{23}, \overline{\sigma}_{12})$ (there can be cuts between $\overline{\rho}_2$ and $\overline{\rho}_3$, and $\overline{\sigma}_1$ and $\overline{\sigma}_2$, but these cuts will increase the Maslov index by one). Under the interpretation, the domain is an annulus with one boundary consisting of two segments of $\boldsymbol{\alpha}$ curves and two segments of $\boldsymbol{\beta}$ curves, and another boundary consisting of $\boldsymbol{\alpha}$ curve only. In the sense of~\cite[Lemma 9.4]{OZ02}, such an annulus cannot allow a holomorphic involution that interchanges one boundary with another, carrying $\boldsymbol{\alpha}$ curves to $\boldsymbol{\alpha}$ curves and $\boldsymbol{\beta}$ curves to $\boldsymbol{\beta}$ curves. Thus, the moduli space $\mathcal{M} ( \mathbf{x_{2n-2} y_2}, \mathbf{x_1 y_1} ; \overline{\rho}_{23}, \overline{\sigma}_{12})$ cannot give a nontrivial differential. Domains such as $Q_2 + Q_3 + Q_4 + P_1 + P_2$ or $Q_2 + Q_3 + Q_4 + R_1 + P_2$ can be considered similarly to $Q_2 + Q_3 + Q_4$. In fact, they do not give the nontrivial differential as long as the shape of the domain is topologically equivalent to $Q_2 + Q_3 + Q_4$. \\

There are two domains possibly giving a nontrivial differential; they are $Q_2 + Q_3 + P_1 + \cdots + P_{2n-3} + Q_4$ and $Q_2 + Q_3 + R_1 + P_2 + \cdots + R_{2n-3} + Q_4$. We will consider the domain $Q_2 + Q_3 + P_1 + \cdots + P_{2n-3} + Q_4$ first. It has three interpretations. Each of the interpretations comes from the choice of cuts made on the boundary of the domain. Cuts are allowed where the domain has $270^{\circ}$ or $180^{\circ}$ corners, or a point on the boundary intersecting $\boldsymbol{\alpha}$ curve. Thus, the domain $Q_2 + Q_3 + P_1 + \cdots + P_{2n-3} + Q_4$ has two points that possibly allow cuts; a point between $\rho_1$ and $\rho_2$, and a point between $\sigma_2$ and $\sigma_3$. Of course, it may not have any cuts at all. We list the moduli spaces of these interpretations as below. 
\begin{itemize}
  \item $\mathcal{M} ( \mathbf{a y_{2n-1}}, \mathbf{ a y_1} ; \overline{\rho}_3, \overline{\rho}_2, \overline{\sigma}_{12} )$
  \item $\mathcal{M} ( \mathbf{x_{2n-1} y_{2n-1}}, \mathbf{x_{2n-1} y_1} ; \overline{\rho}_{23}, \overline{\sigma}_2, \overline{\sigma}_1 )$
  \item $\mathcal{M}( \mathbf{x_{2k-1} y_{2n-1} }, \mathbf{x_{2k-1} y_1 } ; \overline{\rho}_{23}, \overline{\sigma}_{12} )$
\end{itemize}
First, we will consider $\mathcal{M}( \mathbf{x_{2k-1} y_{2n-1} }, \mathbf{x_{2k-1} y_1 } ; \overline{\rho}_{23}, \overline{\sigma}_{12} )$.

\begin{lem}
\label{thm:nonex1}
Modulo two count of the moduli space $\mathcal{M}( \mathbf{x_{2k-1} y_{2n-1} }, \mathbf{x_{2k-1} y_1 } ; \overline{\rho}_{23}, \overline{\sigma}_{12} )$ equals zero.
\end{lem}
\begin{proof}
We will compute the signed number of the moduli space by considering the following $\mathcal{A}_{\infty}$ compatibility condition.
\begin{eqnarray*}
  0 & = & m ( m ( \mathbf{x_{2k-1} y_{2n-1} } ), \overline{\rho}_2, \overline{\rho}_3, \overline{\sigma}_{12} ) +  m ( m ( \mathbf{ x_{2k-1} y_{2n-1} }, \overline{\rho}_2 ), \overline{\rho}_3, \overline{\sigma}_{12} ) \\
  & & + m ( m ( \mathbf{x_{2k-1} y_{2n-1} }, \overline{\sigma}_{12} ), \overline{\rho}_2, \overline{\rho}_3 ) + m ( \mathbf{x_{2k-1} y_{2n-1} } , \mu ( \overline{\rho}_2, \overline{\rho}_3 ) , \overline{\sigma}_{12} ) \\
  & & + m ( m ( \mathbf{ x_{2k-1} y_{2n-1} }, \overline{\rho}_2, \overline{\sigma}_{12} ),  \overline{\rho}_3 ) + m ( m ( \mathbf{x_{2k-1} y_{2n-1} }, \overline{\rho}_2, \overline{\rho}_3, \overline{\sigma}_{12} ) )
\end{eqnarray*}
The right hand side of the equation above consists of six terms. The second term vanishes because $m(\mathbf{x_{2k-1} y_{2n-1}}, \overline{\rho}_2)$ does not have the algebra element $\overline{\sigma}_2$ (note that domain $Q_2$ is adjacent to $\overline{\rho}_2$ and $\overline{\sigma}_2$). Similarly, the third term vanishes since $m(\mathbf{x_{2k-1} y_{2n-1}}, \overline{\sigma}_{12})$ has $\overline{\sigma}_{12}$ as its input but lacks $\overline{\rho}_2$. The last term also vanishes because the Maslov index is not one. Replacing $\mu (\overline{\rho}_2, \overline{\rho}_3) = \overline{\rho}_{23}$ and $m ( \mathbf{x_{2k-1} y_{2n-1} } ) = \mathbf{x_{2n-2} y_{2k}} + \mathbf{x_{2k} y_{2n-2} }$, the above equation is reduced as follows.
\begin{eqnarray*}
  0 & = & m ( \mathbf{x_{2n-2} y_{2k}}, \overline{\rho}_2, \overline{\rho}_3, \overline{\sigma}_{12} ) + m ( \mathbf{x_{2k} y_{2n-2}}, \overline{\rho}_2, \overline{\rho}_3, \overline{\sigma}_{12} ) \\
  & & + m ( \mathbf{x_{2k-1} y_{2n-1}} , \overline{\rho}_{23}, \overline{\sigma}_{12} ) + m ( m ( \mathbf{x_{2k-1} y_{2n-1}} , \overline{\rho}_2, \overline{\sigma}_{12} ), \overline{\rho}_3 ).
\end{eqnarray*}
The first term on the right hand side corresponds to the moduli space
\begin{displaymath}
\mathcal{M}(\mathbf{x_{2n-2} y_{2k}}, \mathbf{x_{2k-1} y_1} ; \overline{\rho}_2, \overline{\rho}_3, \overline{\sigma}_{12} ),
\end{displaymath}
whose Maslov index is not one. The second term also vanishes because any domain containing $Q_2 + Q_3 + Q_4$ cannot have corners that contain $\mathbf{x_{2k}}$ and $\mathbf{y_{2n-2}}$. The last term also vanishes because of the following reason; the moduli space $\mathcal{M} ( \mathbf{x_{2k-1} y_{2n-1} }, \mathbf{a y_{2k} }; \overline{\rho}_2, \overline{\sigma}_{12} )$ has no holomorphic representative since the domain is an annulus and does not allow holomorphic involution, so $m ( \mathbf{x_{2k-1} y_{2n-1}} , \overline{\rho}_2, \overline{\sigma}_{12} ) = 0$. Hence, $m (  \mathbf{x_{2k-1} y_{2n-1}} , \overline{\rho}_{23}, \overline{\sigma}_{12} ) = 0$ and $\sharp \mathcal{M} (\mathbf{ x_{2k-1} y_{2n-1} }, \mathbf{x_{2k-1} y_1} ; \overline{\rho}_{23}, \overline{\sigma}_{12} ) = 0$.
\end{proof}

The second interpretation is $\mathcal{M} ( \mathbf{x_{2n-1} y_{2n-1}}, \mathbf{x_{2n-1} y_1} ; \overline{\rho}_{23}, \overline{\sigma}_2, \overline{\sigma}_1 )$. The domain is an annulus; each boundary consists of one $\boldsymbol{\alpha}$ curve segment and one $\boldsymbol{\beta}$ curve segment. The modulo two count of the moduli space can be computed by a similar computation above.
\begin{lem}
\label{thm:nonex2}
Modulo two count of moduli space $\mathcal{M} ( \mathbf{x_{2n-1} y_{2n-1}}, \mathbf{x_{2n-1} y_1} ; \overline{\rho}_{23}, \overline{\sigma}_2, \overline{\sigma}_1 )$ is one.
\end{lem}
\begin{proof}
Again, we will consider the $\mathcal{A}_{\infty}$-compatibility relation as below.
\begin{eqnarray*}
0 & = & m^2 ( \mathbf{x_{2n-1} y_{2n-1} }, \overline{\rho}_2,
\overline{\rho}_3, \overline{\sigma}_2, \overline{\sigma}_1 ) \\
& = & m ( m ( \mathbf{x_{2n-1} y_{2n-1} } ) , \overline{\rho}_2,
\overline{\rho}_3, \overline{\sigma}_2, \overline{\sigma}_1 ) + m (
\mathbf{x_{2n-1} y_{2n-1} }, \mu(\overline{\rho}_2,
\overline{\rho}_3), \overline{\sigma}_2, \overline{\sigma}_1 ) \\
& + & m ( m ( \mathbf{x_{2n-1} y_{2n-1}}, \overline{\rho}_2,
\overline{\sigma}_2 ) , \overline{\rho}_3, \overline{\sigma}_1 ) + m
( m (\mathbf{x_{2n-1} y_{2n-1}}, \overline{\rho}_2,
\overline{\rho}_3 ), \overline{\sigma}_2, \overline{\sigma}_1 ) \\
& + & m ( m (\mathbf{x_{2n-1} y_{2n-1}}, \overline{\sigma}_2,
\overline{\sigma}_1 ), \overline{\rho}_2, \overline{\rho}_3 ) + m (
m (\mathbf{x_{2n-1} y_{2n-1}}, \overline{\rho}_2, \overline{\rho}_3
, \overline{\sigma}_2 ), \overline{\sigma}_1 ) \\
& + & m ( m (\mathbf{x_{2n-1} y_{2n-1}}, \overline{\rho}_2,
\overline{\sigma}_2 , \overline{\sigma}_1 ), \overline{\rho}_3 ) + m
( m ( \mathbf{x_{2n-1} y_{2n-1} }, \overline{\rho}_2,
\overline{\rho}_3, \overline{\sigma}_2, \overline{\sigma}_1 ) )
\end{eqnarray*}
$m(\mathbf{x_{2n-1} y_{2n-1}})=0$ since there is no provincial domain connecting $\mathbf{x_{2n-1} y_{2n-1}}$, the first term on the right hand side vanishes. The fourth term also vanishes because $m( \mathbf{x_{2n-1} y_{2n-1}}, \overline{\rho}_2, \overline{\rho}_3 ) =0$ (domain $Q_2$ is adjacent to both
$\overline{\rho}_2$ and $\overline{\sigma}_2$). By the same reason, the fifth term vanishes.

$m(\mathbf{x_{2n-1} y_{2n-1}}, \overline{\rho}_2, \overline{\rho}_3, \overline{\sigma}_2 )$ in the sixth term does not represent a domain with four corners. Recall that a domain that involves $\overline{\rho}_2$ and $\overline{\rho}_3$ must have $\overline{\sigma}_1$. Thus, the sixth term vanishes. Similarly, the seventh term also vanishes. $m ( \mathbf{x_{2n-1} y_{2n-1} }, \overline{\rho}_2, \overline{\rho}_3, \overline{\sigma}_2, \overline{\sigma}_1 ) = 0$ when considering the Maslov index.

Then the above compatibility relation is reduced to,
\begin{displaymath}
  m ( \mathbf{x_{2n-1} y_{2n-1}}, \overline{\rho}_{23}, \overline{\sigma}_2, \overline{\sigma}_1 ) + m ( m ( \mathbf{x_{2n-1} y_{2n-1}}, \overline{\rho}_2, \overline{\sigma}_2 ), \overline{\rho}_3, \overline{\sigma}_1 ) = 0
\end{displaymath}
Note that the second term on the left hand side equals $\mathbf{x_{2n-1} y_1} + \mathbf{x_1 y_{2n-1}}$. This implies modulo two count of the moduli spaces $\mathcal{M} ( \mathbf{x_{2n-1} y_{2n-1}}, \mathbf{x_{2n-1} y_1} ; \overline{\rho}_{23}, \overline{\sigma}_2, \overline{\sigma}_1 )$ and $\mathcal{M} ( \mathbf{x_{2n-1} y_{2n-1}}, \mathbf{x_1 y_{2n-1}} ; \overline{\rho}_{23}, \overline{\sigma}_2, \overline{\sigma}_1 )$ equal one.
\end{proof}

However, idempotents of the type-$DD$ module prohibit a nontrivial differential from moduli spaces considered above.
Explicitly,
\begin{eqnarray*}
  \delta^1 (\mathbf{ x_{2n-1} y_{2n-1} } ) & = & \rho_{12} \otimes
  \sigma_{23} \otimes \mathbf{ x_{2n-1} y_1 } + \cdots \\
  & = & \rho_{12} \iota_{1} \otimes \sigma_{23} \otimes \mathbf{ x_{2n-1} y_1 } + \cdots =
  \rho_{12} \otimes \sigma_{23} \otimes \iota_{1} \mathbf{ x_{2n-1} y_1 } + \cdots
\end{eqnarray*}
Recall that $\iota_1 \mathbf{ x_{2n-1} y_1 } =0$ since $\mathbf{ x_{2n-1} y_1 }$ occupies $\alpha_1^{a,L}$ and the idempotent $\iota_1$ also occupies the same $\alpha$-arc. \\

The third interpretation is $\mathcal{M} ( \mathbf{a y_{2n-1}}, \mathbf{ a y_1} ; \overline{\rho}_3, \overline{\rho}_2, \overline{\sigma}_{12} )$. This is again an annulus and one of its boundaries has two $\boldsymbol{\alpha}$ curve segments and two $\boldsymbol{\beta}$ curve segments, thus it cannot give a nontrivial differential either. \\

Next, we will consider domain $Q_2 + Q_3 + R_1 + P_2 + \cdots + R_{2n-3} + Q_4$. Possible cuts may arise from a point between $\sigma_2$ and $\sigma_3$. The possible interpretations are
\begin{itemize}
  \item $\mathcal{M} ( \mathbf{x_{2n-1} y_{2n-1} }, \mathbf{x_1 y_{2n-1} } ; \overline{\rho}_{23}, \overline{\sigma}_2, \overline{\sigma}_1 )$
  \item $\mathcal{M} ( \mathbf{x_{2n-1} y_{2k-1} }, \mathbf{x_1 y_{2k-1} } ; \overline{\rho}_{23}, \overline{\sigma}_{12} )$
\end{itemize}
By the above lemma, the modulo two count of the first moduli space is one, but because of idempotents, it cannot give a nontrivial contribution to the differential. The second moduli space has modulo two count zero by a similar computation in Lemma~\ref{thm:nonex1} or Lemma~\ref{thm:nonex2}. \\

\textbf{Algebra element containing $\rho_{23}$.} Roughly speaking, the domains that possibly contribute to the algebra element $\rho_{23}$ is obtained by adding regions to the domain $Q_1 +Q_2$ so that the resulting domain has at most four corners. \\

We will consider these domains by classifying the domains into three cases. \\

\textbf{Case 1.} We will first consider the following annular domains.
\begin{eqnarray*}
  & & Q_1 + Q_2, \\
  & & Q_1 + Q_2 + Q_4 + R_{2n-3} + P_{2n-3} + R_{2n-4}, \\
  & & Q_1 + Q_2 + Q_4 + R_{2n-3} + P_{2n-3} + R_{2n-4} + \cdots + P_{2k-1} + R_{2k}, \\
  & & \qquad \vdots 
\end{eqnarray*}
Basically, these domains are obtained by adding even number of regions to the top and bottom of $Q_1 + Q_2$. \\

We will first consider domain $Q_1+Q_2$. The domain can be interpreted as $\mathcal{M}(\mathbf{a y_{2n-2}}, \mathbf{a b}; \overline{\rho}_{12}, \overline{\sigma}_2)$. Again, modulo two count of the moduli space can be computed by using $\mathcal{A}_{\infty}$-relation of $m^2(\mathbf{a y_{2n-2}}, \overline{\rho}_1, \overline{\rho}_2, \overline{\sigma}_2)$. Recall that $m (\mathbf{a y_{2n-2}}, \overline{\rho}_1) = \mathbf{x_{2n-1} y_{2n-1}}$ and $m (\mathbf{x_{2n-1} y_{2n-1}}, \overline{\rho}_2, \overline{\sigma}_2) = \mathbf{a b}$ since the associated domains are rectangular.
\begin{eqnarray*}
  0 & = & m( m(\mathbf{a y_{2n-2}}, \overline{\rho}_1), \overline{\rho}_2,
  \overline{\sigma}_2) + m(\mathbf{a y_{2n-2}}, (\overline{\rho}_1, \overline{\rho}_2),
  \overline{\sigma}_2) + m( m(\mathbf{a y_{2n-2}},  \overline{\sigma}_2),
  \overline{\rho}_1, \overline{\rho}_2)\\
  & = & \mathbf{a b} + m(\mathbf{a y_{2n-2}}, \overline{\rho}_{12},
  \overline{\sigma}_2) + m( m(\mathbf{a y_{2n-2}},  \overline{\sigma}_2),
  \overline{\rho}_1, \overline{\rho}_2)
\end{eqnarray*}
The last term on the right hand side equals zero because $m( \mathbf{a y_{2n-2}}, \overline{\sigma}_2 ) =0$ (domain $Q_2$ is adjacent to Reeb chords $\overline{\rho}_2$ and $\overline{\sigma}_2$). This implies $m(\mathbf{a y_{2n-2}}, \overline{\rho}_{12}, \overline{\sigma}_2) = \mathbf{a b}$, hence $\sharp \mathcal{M}(\mathbf{a y_{2n-2}}, \mathbf{a b}; \overline{\rho}_{12}, \overline{\sigma}_2) = 1$.

\begin{rem}
An annulus domain of such kind (i.e, an outside boundary consisting of both $\boldsymbol{\alpha}$ and $\boldsymbol{\beta}$ curves and an inside boundary of $\boldsymbol{\alpha}$ curve only, including a cut on the inside boundary) always admits a holomorphic representative; since we are free to choose the length of the cut starting from the point so that the annulus admits a biholomorphic involution of it, again in the sense of \cite[Lemma 9.4]{OZ02}.
\label{rem:involution}
\end{rem}

The moduli space $\mathcal{M}(\mathbf{a y_{2n-2}}, \mathbf{a b}; \overline{\rho}_{12}, \overline{\sigma}_2) = \mathcal{M}(\mathbf{a y_{2n-2}}, \mathbf{a b}; \rho_{23}, \sigma_2)$ corresponding to $\rho_{23} \otimes \sigma_2 \otimes \mathbf{ab}$ term occurs in $\delta^1(\mathbf{a y_{2n-2}})$ in $\widehat{CFDD}$. However, the right hand side is zero because of the idempotents. \\

Likewise, domains $Q_1 + Q_2 + Q_4 + R_{2n-3} + P_{2n-3} + R_{2n-2}$, $Q_1 + Q_2 + Q_4 + R_{2n-3} + P_{2n-3} + R_{2n-2} + \cdots + P_{2k-1} + R_{2k}$, $\cdots$ allow the following interpretations.
\begin{itemize}
  \item $\mathcal{M}(\mathbf{a y_{2j}},\mathbf{a y_{2j+2}};
  \overline{\rho}_{12}, \overline{\sigma}_2, \overline{\sigma}_1 )$
  \item $\mathcal{M} (\mathbf{a y_{2j}},\mathbf{a y_{2j+2}};
  \overline{\rho}_{12}, \overline{\sigma}_{12} )$
\end{itemize}
The signed number of the first interpretation $\mathcal{M}(\mathbf{a y_{2j}},\mathbf{a y_{2j+2}}; \overline{\rho}_{12}, \overline{\sigma}_2, \overline{\sigma}_1 )$ is one modulo two by similar reasons described in Remark~\ref{rem:involution}. These contribute to the differential between generators $\mathbf{a y_{2j}}$ and $\mathbf{a y_{2j+2}}$ with an algebra element containing $\rho_{23}$, but all going to zero because of idempotents as well as $\mathcal{M}(\mathbf{a b},\mathbf{a y_2}; \overline{\rho}_{12}, \overline{\sigma}_ 3, \overline{\sigma}_2, \overline{\sigma}_1 )$. \\

The signed number of the second interpretation $\mathcal{M} (\mathbf{a y_{2j}},\mathbf{a y_{2j+2}}; \overline{\rho}_{12}, \overline{\sigma}_{12} )$ is zero modulo two. It can be proved by considering the following $\mathcal{A}_{\infty}$-relation.
\begin{eqnarray*}
  0 & =& m(m(\mathbf{a y_{2j}}, \overline{\rho}_{12}, \overline{\sigma}_1,
  \overline{\sigma}_2)) + m( m(\mathbf{a y_{2j}},
  \overline{\sigma}_1 ), \overline{\rho}_{12}, \overline{\sigma}_2 )
  \\
  & +& m( m (\mathbf{a y_{2j}},
  \overline{\rho}_{12} ), \overline{\sigma}_1 , \overline{\sigma}_2 )
  + m( m (\mathbf{a y_{2j}},
  \overline{\rho}_{12} , \overline{\sigma}_1) , \overline{\sigma}_2 )
  + m( \mathbf{a y_{2j}},
  \overline{\rho}_{12} , (\overline{\sigma}_1 ,
  \overline{\sigma}_2))
\end{eqnarray*}
$m(\mathbf{a y_{2j}}, \overline{\rho}_{12}, \overline{\sigma}_1, \overline{\sigma}_2) = 0$ since Maslov index is not one. $m (\mathbf{a y_{2j}}, \overline{\rho}_{12} )$ and $m (\mathbf{a y_{2j}}, \overline{\rho}_{12} , \overline{\sigma}_1)$ equal zero because $\overline{\sigma}_2$ was not involved and there is no such domain corresponding to these interpretations. $m( \mathbf{a y_{2j}},  \overline{\sigma}_1 ) =0$ is clear from the diagram.  Thus, the last term $m( \mathbf{a y_{2j}}, \overline{\rho}_{12} , (\overline{\sigma}_1 , \overline{\sigma}_2) ) = m( \mathbf{a y_{2j}}, \overline{\rho}_{12}, \overline{\sigma}_{12} )$ equals zero, too. \\

\textbf{Case 2.} Next, we will consider the following domains.
\begin{eqnarray*}
  & & Q_1 + Q_2 + P_{2n-3} + R_{2n-4} + \cdots + R_{2k} + P_{2k-1}, \\
  & & Q_1 + Q_2 + P_{2n-3} + R_{2n-4} + \cdots + R_{2k} + P_{2k-1}  \\
  & & \quad + \ Q_4 + R_{2n-3} + P_{2n-4} + \cdots + P_{2l} + R_{2l-1}, \\
  & & Q_1 + Q_2 + P_{2n-3} + R_{2n-4} + \cdots + P_{2k+1} + R_{2k}  \\
  & & \quad + \ Q_4 + R_{2n-3} + P_{2n-4} + \cdots + R_{2l+1} + P_{2l}, \\
  & & Q_1 + Q_2 + P_{2n-3} + R_{2n-4} + \cdots + R_2 + P_1 + Q_5 \\
  & & \quad + \ Q_4 + R_{2n-3} + P_{2n-4} + \cdots + R_{2l+1} + P_{2l}.
\end{eqnarray*}

These domains are obtained by adding a topologically rectangular domain containing $Q_1 + Q_2$ and another rectangular domain containing $Q_4$. \\ 

The first domain can have a cut at a point between $\rho_2$ and $\rho_3$. The interpretation 
\begin{displaymath}
\mathcal{M} (\mathbf{x_{2k-1} y_{2n-1}} , \mathbf{x_{2k} b} ; \overline{\rho}_2, \overline{\rho}_1,
\overline{\sigma}_2 )
\end{displaymath}
is essentially a rectangle so modulo two count of the corresponding moduli space is one. The second domain can have cuts at two different points; a point between $\rho_2$ and $\rho_3$, and a point between $\sigma_2$ and $\sigma_3$. Considering the interpretation that has only one cut, the domain is an annulus with one of its boundary consisting of two $\boldsymbol{\alpha}$ curve segments and two $\boldsymbol{\beta}$ curve segments, which does not allow any holomorphic representative. If the interpretation has both of the cuts, then it is also a rectangular domain of the moduli space $\mathcal{M} ( \mathbf{x_{2k-1} y_{2l-1}}, \mathbf{x_{2k} y_{2l}} ; \overline{\rho}_2, \overline{\rho}_1, \overline{\sigma}_2, \overline{\sigma}_1 )$. Dualizing them, they yield a nontrivial differential of algebra elements $\rho_{23} \otimes \sigma_2$ and $\rho_{23} \otimes \sigma_{23}$ for the type-$D$ structure map $\delta^1$ in $\widehat{CFDD}$.
\begin{rem}
Both of the domains considered above have interpretations without any cut. However, those interpretations do not have a holomorphic representative. For example, modulo two count of the moduli space $\mathcal{M} ( \mathbf{x_{2k-1} y_{2l-1}}, \mathbf{x_{2k} y_{2l}} ; \overline{\rho}_{12}, \overline{\sigma}_{12} )$ equals zero by considering a similar $\mathcal{A}_{\infty}$ relation discussed in Lemma~\ref{thm:nonex1} and Lemma~\ref{thm:nonex2}.
\end{rem}
The third domain has almost the same interpretation; the only meaningful interpretation is
\begin{displaymath}
\mathcal{M} ( \mathbf{x_{2l} y_{2k}}, \mathbf{x_{2l+1} y_{2k+1}} ; \overline{\rho}_2, \overline{\rho}_1, \overline{\sigma}_2, \overline{\sigma}_1 ).
\end{displaymath}
Again, this interpretation is rectangular and modulo two count of the moduli space is one. \\

The last domain has two interpretations with Maslov index one. They are
\begin{displaymath}
\mathcal{M} ( \mathbf{x_{2l} b}, \mathbf{x_{2l+1} y_1} ; \overline{\rho}_2, \overline{\rho}_1, \overline{\sigma}_3, \overline{\sigma}_2, \overline{\sigma}_1 )
\end{displaymath} 
and
\begin{displaymath}
\mathcal{M} ( \mathbf{x_{2l} b}, \mathbf{x_{2l+1} y_1} ; \overline{\rho}_2, \overline{\rho}_1, \overline{\sigma}_{123} ).
\end{displaymath} 
The first interpretation is clearly a rectangle. However, the second one is topologically a punctured torus. To count the signed number of the moduli space, we investigate the $\mathcal{A}_{\infty}$-relation $m^2 ( \mathbf{ x_{2l} b }, \overline{\rho}_2, \overline{\rho}_1, \overline{\sigma}_{12}, \overline{\sigma}_3 ) = 0$. 
\begin{lem}
Modulo two count of the moduli space $\mathcal{M} ( \mathbf{x_{2l} b}, \mathbf{x_{2l+1} y_1} ; \overline{\rho}_2, \overline{\rho}_1, \overline{\sigma}_{123} )$ equals one.
\label{lem:puncturedtorus}
\end{lem}
\begin{proof}
Disregarding all terms that equal to zero, the relation is reduced to
\begin{displaymath}
m ( \mathbf{ x_{2l} b }, \overline{\rho}_2, \overline{\rho}_1, \overline{\sigma}_{123} ) + m ( m ( \mathbf{ x_{2l} b } , \overline{\rho}_2, \overline{\rho}_1, \overline{\sigma}_{12} ) , \overline{\sigma}_3 ) =0 .
\end{displaymath}
$m ( \mathbf{ x_{2l} b } , \overline{\rho}_2, \overline{\rho}_1, \overline{\sigma}_{12} ) = \mathbf{x_{2l +2} b} $ because the corresponding domain is an annulus as in Lemma~\ref{rem:involution}. Thus, the relation is reduced to
\begin{displaymath}
m ( \mathbf{ x_{2l} b }, \overline{\rho}_2, \overline{\rho}_1, \overline{\sigma}_{123} ) + m ( \mathbf{ x_{2l+2} b } , \overline{\sigma}_3 ) =0 .
\end{displaymath} 
The second term of the right hand side is clearly $\mathbf{x_{2l+1} y_1 } + \mathbf{x_1 y_{2l+1}}$. This implies modulo two count of the moduli space
\begin{displaymath}
\mathcal{M} ( \mathbf{x_{2l} b}, \mathbf{x_{2l+1} y_1} ; \overline{\rho}_2, \overline{\rho}_1, \overline{\sigma}_{123} )
\end{displaymath}
equals one. (Clearly this lemma also proves that modulo two count of 
\begin{displaymath}
\mathcal{M} ( \mathbf{x_{2l} b}, \mathbf{x_1 y_{2l+1} } ; \overline{\rho}_2, \overline{\rho}_1, \overline{\sigma}_{123} )
\end{displaymath}
equals one, too.)
\end{proof}
The two interpretations of the last domain result in the two same terms $\rho_{23} \otimes \sigma_{123} \otimes \mathbf{x_{2l+1} y_1}$ in $\widehat{CFDD}$ module; therefore they do not contribute to a nontrivial differential. \\

\begin{figure}
  \centering
  \includegraphics[width=1\textwidth]{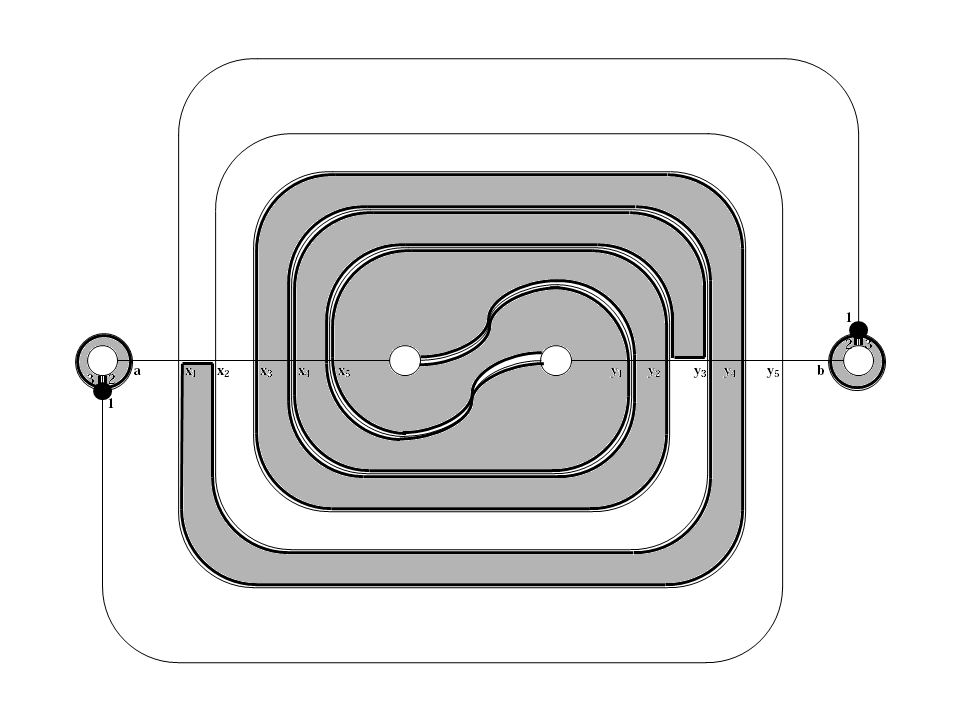}
  \caption{A diagram of (2,6) torus link complement. The shaded region is a domain obtained by adding a rectangular domains to $Q_2$. This domain corresponds to a differential from $\mathbf{x_1 y_3}$ to $\mathbf{x_2 y_2}$. Cutting along the bold curve on the boundary of the domain, the domain turns out to be rectangular.}
  \label{fig:26domain}
\end{figure}

\textbf{Case 3.} Domains that possibly contribute to a differential with an algebra element containing $\rho_{23}$ are obtained by adding domains to the top of $Q_1 + Q_2$. That is, we add $2j-1$ domains, $j=1, \cdots, n-1$ on the top and the resulting domain is $R_{2n-2j-1} + \cdots + R_{2n-3} + Q_1 + Q_2$. The only possible interpretation is $\mathcal{M} (\mathbf{x_{2n-1} y_{2n-2j-1}}, \mathbf{x_{2n-2j} b} ; \overline{\rho}_{12}, \overline{\sigma}_2 )$. It does not allow any holomorphic representative because the domain does not allow any holomorphic involution interchanging two boundaries. \\

Likewise, we shall consider domains obtained by adding domains to $Q_2$ on the top and bottom. Consider a domain
\begin{displaymath}
Q_1 + Q_2 + Q_4 + (R_{2n-k}+ \cdots R_{2n-3}) + (P_{2n-l} + \cdots + P_{2n-3}).
\end{displaymath}
The domain is obtained by adding $k-2$ domains on the top and $l-2$ domains on the bottom of $Q_1 +Q_2 + Q_4$ ($k$ and $l$ should have the same parity). If $k=l$, then the resulting domain is the domain that we have considered in Case 2. (See the bottom right of Figure~\ref{fig:cases}.) If $k \neq l$, then three interpretations are possible. The first is $\mathcal{M} (\mathbf{x_{2n-l-2} y_{2n-k-2}}, \mathbf{x_{2n-k-1} y_{2n-l-1}} ; \overline{\rho}_{12}, \overline{\sigma}_{12} )$. This is a genus two domain, and modulo two count of this moduli space is zero by a similar reason given in Lemma~\ref{thm:nonex1}. The second and third interpretations are 
\begin{displaymath}
\mathcal{M} (\mathbf{x_{2n-l-2} y_{2n-k-2}}, \mathbf{x_{2n-k-1} y_{2n-l-1}} ; \overline{\rho}_{12}, \overline{\sigma}_2, \overline{\sigma}_1 )
\end{displaymath}
and
\begin{displaymath}
\mathcal{M} (\mathbf{x_{2n-l-2} y_{2n-k-2}}, \mathbf{x_{2n-k-1} y_{2n-l-1}} ; \overline{\rho}_2, \overline{\rho}_1, \overline{\sigma}_{12} ).
\end{displaymath}
These are both annular interpretations, and they do not have any holomorphic representative because they do not allow a holomorphic involution. \\

Lastly, if $k=2n-2$, then, the domain contains $Q_5$. Then this domain has the following two interpretations.
\begin{itemize}
  \item $\mathcal{M} ( \mathbf{x_{2l-2} b}, \mathbf{x_1 y_{2l-1} } ; \overline{\rho}_2, \overline{\rho}_1, \overline{\sigma}_{123} )$
  \item $\mathcal{M} ( \mathbf{x_{2l-2} b}, \mathbf{x_1 y_{2l-1} } ; \overline{\rho}_2, \overline{\rho}_1, \overline{\sigma}_3 , \overline{\sigma}_{12} )$
\end{itemize}
The signed number of the moduli space of the first interpretation was proved to be one modulo two by Lemma~\ref{lem:puncturedtorus}. The signed number of the second interpretation is not one because it does not allow a holomorphic involution either. \\ 

To sum up, the differentials that give the algebra element containing $\rho_{23}$ are listed below.
\begin{itemize}
  \item $\mathbf{x_i y_j} \mapsto \rho_{23} \otimes \sigma_{23} \otimes \mathbf{x_{i+1} y_{j+1}} $ if $i, j \neq 2n-1$;
  \item $\mathbf{x_i y_{2n-1}} \mapsto \rho_{23} \otimes \sigma_2 \otimes \mathbf{x_{i+1} b}$ if $j =2n-1$ and $i = 1, 3, \cdots, 2n-3$;
  \item $\mathbf{x_i b} \mapsto \rho_{23} \otimes \sigma_{123} \otimes \mathbf{x_1 y_{i+1} }$.
\end{itemize}

\begin{figure}
  \centering
  \includegraphics[width=1\textwidth]{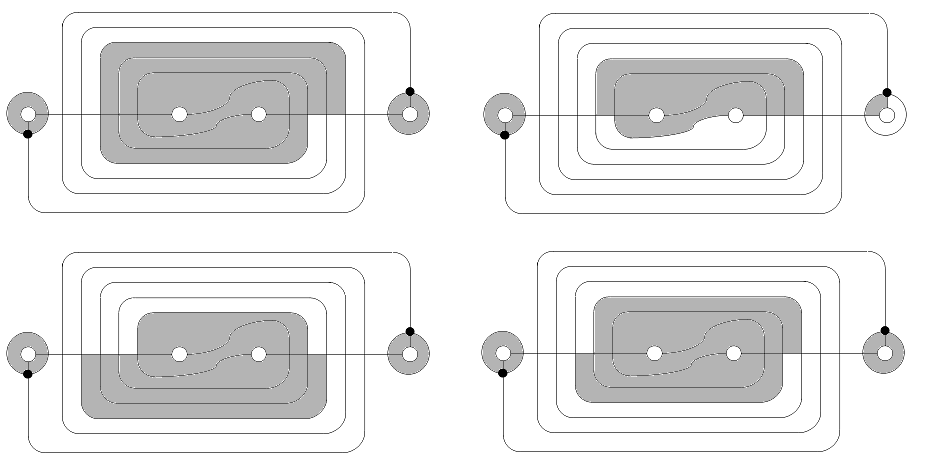}
  \caption{The above diagram shows examples of obtaining non-rectangular domains of $(2,6)$-torus link. Top left can be interpreted as an annular domain, but it cannot give a nontrivial differential due to idempotents. Top right is obtained by adding a domain to $Q_2$ on the top, but its only possible interpretation does not allow any holomorphic representative. Bottom left and bottom right are obtained by  adding domains to $Q_2$ on the top and bottom. If the number of regions attached on the top is not equal to the number of regions attached on the bottom, it has two interpretations; and they do not allow a holomorphic representative either (bottom left). If two numbers are equal, then the domain gives a nontrivial differential. (This case was previously considered. See Figure~\ref{fig:26domain}.) }
  \label{fig:cases}
\end{figure}

\textbf{Algebra element contains $\rho_{123}$.} Domains that possibly contribute to the algebra element $\rho_{123}$ are listed below.
\begin{eqnarray*}
  & & (Q_1 + Q_2 + Q_3 + Q_4 + R_{2n-3}) + R_1 + P_2 + R_3 + \cdots + P_{2n-4},  \\
  & & (Q_1 + Q_2 + Q_3 + Q_4) + P_1 + \cdots P_{2n-3}, \\[0.2cm]
  & & (Q_1 + Q_2 + Q_3 + Q_4 + R_{2n-3} + P_{2n-3} + R_{2n-4} + P_{2n-4} + R_{2n-5}) + R_1 + P_2 + R_3 + \cdots + P_{2n-6}, \\
  & & (Q_1 + Q_2 + Q_3 + Q_4 + R_{2n-3} + P_{2n-4} + R_{2n-4} + P_{2n-4} + R_{2n-5}) + P_1 +
  \cdots P_{2n-5}, \\[0.2cm]
  & & \quad \vdots \\
  & & Q_1 + \cdots + Q_5 + P_1 + \cdots + P_{2n-3} + R_1 + \cdots + R_{2n-3}. \\
\end{eqnarray*}

Each of these domains are obtained by adding a rectangular domain containing a region adjacent to $\rho_1$ to the annular domain listed in algebra element containing $\rho_{23}$. \\

We will investigate the first domain. As before, we will list all possible interpretations.
\begin{itemize}
  \item $\mathcal{M} ( \mathbf{a y_{2n-2}}, \mathbf{x_1 y_{2n-1}} ; \overline{\rho}_{123}, \overline{\sigma}_2,   \overline{\sigma}_1)$
  \item $\mathcal{M} ( \mathbf{a y_{2n-2}}, \mathbf{x_1 y_{2n-1}} ; \overline{\rho}_3, \overline{\rho}_2, \overline{\rho}_1, \overline{\sigma}_2, \overline{\sigma}_1 )$
  \item $\mathcal{M} ( \mathbf{a y_{2n-2}}, \mathbf{x_1 y_{2n-1}} ; \overline{\rho}_{23}, \overline{\rho}_1, \overline{\sigma}_2, \overline{\sigma}_1 )$
  \item $\mathcal{M} ( \mathbf{a y_{2n-2}}, \mathbf{x_1 y_{2n-1}} ; \overline{\rho}_3, \overline{\rho}_{12}, \overline{\sigma}_2, \overline{\sigma}_1 )$
\end{itemize}
The third interpretation is an annulus whose outer boundary has two $\boldsymbol{\alpha}$ curve segments and two $\boldsymbol{\beta}$ curve segments; thus it does not have a holomorphic representative. The fourth interpretation cannot give a nontrivial contribution either because of the $\mathcal{A}_{\infty}$-module compatibility relation. On the other hand, the second interpretation is a rectangular one; it allows a holomorphic representative and its modulo two count of the moduli space is one. The first interpretation also has a moduli space with modulo two count one by the same analysis in Lemma~\ref{lem:puncturedtorus}. Again, the first and second interpretations will result in the same term after dualizing to $\widehat{CFDD}$. The sum of these two terms equals zero, so this domain actually has no contribution after all. \\

The second domain has two interpretations; $\mathcal{M}( \mathbf{a y_{2n-2}}, \mathbf{x_{2n-1} y_1} ; \overline{\rho}_{123}, \overline{\sigma}_2, \overline{\sigma}_1 )$ and $\mathcal{M}( \mathbf{a y_{2n-2}}, \mathbf{x_{2n-1} y_1} ; \overline{\rho}_3, \overline{\rho}_{12}, \overline{\sigma}_2, \overline{\sigma}_1 )$. The first interpretation was considered in the above computation, and the second interpretation is an annulus whose outer boundary consists of two $\boldsymbol{\alpha}$ curve segments and two $\boldsymbol{\beta}$ curve segments, so there is no holomorphic representative. \\

Similarly, the other domains (except for the last domain) give Whitney disks, and the moduli spaces corresponding to the domains are $\mathcal{M} (\mathbf{a y_{2j}}, \mathbf{x_1 y_{2j+1}} ; \overline{\rho}_{123}, \overline{\sigma}_2, \overline{\sigma}_1 )$, $\mathcal{M} (\mathbf{a y_{2j}}, \mathbf{x_1 y_{2j+1}} ; \overline{\rho}_3, \overline{\rho}_2, \overline{\rho}_1, \overline{\sigma}_2, \overline{\sigma}_1 )$ and $\mathcal{M} (\mathbf{a y_{2j}}, \mathbf{x_{2j+1} y_1} ; \overline{\rho}_{123}, \overline{\sigma}_2, \overline{\sigma}_1 )$. The signed number of each of these moduli spaces is one modulo two. \\

The moduli space of the last domain $Q_1 + \cdots + Q_5 + P_1 + \cdots + P_{2n-3} + Q_1 + \cdots + Q_{2n-3}$ can be interpreted in four ways. The first is $\mathcal{M} (\mathbf{a b, x_1 y_1}; \overline{\rho}_{123}, \overline{\sigma}_{123} )$ whose Maslov index is different from one. The second possible interpretation is
\begin{displaymath}
  \mathcal{M} ( \mathbf{ab}, \mathbf{x_1 y_1} ; \overline{\rho}_{123}, \overline{\sigma}_3, \overline{\sigma}_2, \overline{\sigma}_1 )
\end{displaymath}
$\mathcal{A}_{\infty}$-relation of $m^2(\mathbf{ab}, \overline{\rho}_{12}, \overline{\rho}_3, \overline{\sigma}_3, \overline{\sigma}_2, \overline{\sigma}_1 )$ gives $m( \mathbf{ab}, \overline{\rho}_{123}, \overline{\sigma}_3, \overline{\sigma}_2, \overline{\sigma}_1 ) = \mathbf{x_1 y_1}$, by considering $m(\mathbf{ab}, \overline{\rho}_{12}, \overline{\sigma}_3, \overline{\sigma}_2, \overline{\sigma}_1 ) = \mathbf{a y_2}$ and $m( \mathbf{a y_2}, \overline{\rho}_3 ) = \mathbf{x_1 y_1}$. Thus, the modulo two count of the moduli space is one. The third interpretation
\begin{displaymath}
\mathcal{M} ( \mathbf{ab}, \mathbf{x_1 y_1} ; \overline{\rho}_3, \overline{\rho}_2, \overline{\rho}_1, \overline{\sigma}_{123} )
\end{displaymath}
can be done precisely in the same way. The last interpretation is 
\begin{displaymath}
  \mathcal{M} (\mathbf{a b, x_1 y_1}; \overline{\rho}_3,
  \overline{\rho}_2, \overline{\rho}_1, \overline{\sigma}_3,
  \overline{\sigma}_2, \overline{\sigma}_1 )
\end{displaymath}
The existence of a holomorphic curve and its modulo two count is quite clear from the diagram; the domain is essentially rectangular in this interpretation. \\

It is worth mentioning that there are three moduli spaces contributing to $\rho_{123} \sigma_{123} \otimes \mathbf{x_1 y_1}$ term in $\delta^1(\mathbf{ab})$.  \\

To sum up, we have the following nontrivial differentials of the algebra element containing $\rho_{123}$.
\begin{itemize}
  \item $\mathbf{a y_{2k} } \mapsto \rho_{123} \otimes \sigma_{23} \otimes \mathbf{x_{2k+1} y_1 } $;
  \item $\mathbf{a b} \mapsto \rho_{123} \otimes \sigma_{123} \otimes \mathbf{x_1 y_1} $.
\end{itemize}

For the algebra elements containing $\sigma_{12}$, $\sigma_{23}$ and $\sigma_{123}$, those differentials can be computed in a parallel manner by taking advantage of the symmetry of the diagram. \\

We close this section by summarizing the computation to the following proposition.

\begin{prop}
Let $\mathcal{H}$ be a bordered Heegaard diagram of $(2,2n)$-torus link complement in $S^3$ as in Figure~\ref{fig:generaldiagram}. Then, $\widehat{CFDD} ( \mathcal{H} ,0 )$ has the following generators.
\begin{itemize}
  \item $\mathbf{x_i y_j}$, where $1 \leq i,j \leq 2n-1$ and $i = j$ modulo two;
  \item $\mathbf{ab}$;
  \item $\mathbf{a y_k}$, where $k = 2 , 4, \cdots, 2n-2$;
  \item $\mathbf{x_k b}$, where $k = 2, 4, \cdots 2n-2$.
\end{itemize}
Then, the map $\delta^1 : \mathfrak{S} (\mathcal{H} ) \rightarrow \mathcal{A}(- \mathcal{Z}_L) \otimes \mathcal{A}( - \mathcal{Z}_R ) \otimes \mathfrak{S} (\mathcal{H})$ is computed in the following way.
\begin{itemize}
  \item For $\mathbf{x_i y_j}$, if $i, j \neq 2n-1$,
  \begin{displaymath}
  \mathbf{x_i y_j} \mapsto
  \left\{
  \begin{array}{lllll}
    \mathbf{x_{j-1} y_{i+1}} + \mathbf{x_{i+1} y_{j-1}} + \rho_{23} \sigma_{23} \mathbf{x_{i+1} y_{j+1}} \quad \textrm{if $j - i > 2$} \\[0.3pc]
    \mathbf{x_{j+1} y_{i-1}} + \mathbf{x_{i-1} y_{j+1}} + \rho_{23} \sigma_{23} \mathbf{x_{i+1} y_{j+1}} \quad \textrm{if $i - j > 2$} \\[0.3pc]
    \mathbf{x_{i+1} y_{j-1}} + \rho_{23} \sigma_{23} \mathbf{x_{i+1} y_{j+1}} \quad \textrm{if $j - i = 2$} \\[0.3pc]
    \mathbf{x_{i-1} y_{j+1}} + \rho_{23} \sigma_{23} \mathbf{x_{i+1} y_{j+1}} \quad \textrm{if $i - j = 2$} \\[0.3pc]
    \rho_{23} \sigma_{23} \mathbf{x_{i+1} y_{j+1}} \quad \textrm{if $i=j$}. 
  \end{array}
  \right. 
  \end{displaymath}    
  \item If $j =2n-1$ and $i = 1, 3, \cdots, 2n-3$,
  \begin{displaymath}
  \mathbf{x_i y_j} \mapsto
  \left\{
  \begin{array}{ll}
    \mathbf{x_{j-1} y_{i+1}} + \mathbf{x_{i+1} y_{j-1}} + \rho_{23} \sigma_2 \mathbf{x_{i+1} b} \quad \textrm{if $j - i > 2$} \\[0.3pc]
    \mathbf{x_{i+1} y_{j-1}} + \rho_{23} \sigma_2 \mathbf{x_{i+1} b} \quad \textrm{if $i = 2n-3$}.
  \end{array}
  \right. 
  \end{displaymath}
  \item If $i=2n-1$ and $j = 1, 3, \cdots, 2n-3$,
  \begin{displaymath}
  \mathbf{x_i y_j} \mapsto
  \left\{
  \begin{array}{ll}
    \mathbf{x_{j+1} y_{i-1}} + \mathbf{x_{i-1} y_{j+1}} + \rho_2 \sigma_{23} \mathbf{a y_{j+1}} \quad \textrm{if $i - j > 2$} \\[0.3pc]
    \mathbf{x_{j+1} y_{i-1}} + \rho_2 \sigma_{23} \mathbf{a y_{j+1}} \quad \textrm{if $j = 2n-3$}.
  \end{array}
  \right.   
  \end{displaymath}
  \item $\mathbf{x_{2n-1} y_{2n-1}} \mapsto \rho_2 \sigma_2 \mathbf{a b} $.
  \item For $\mathbf{a y_j}$, 
  \begin{displaymath}
  \mathbf{a y_j} \mapsto
  \left\{
  \begin{array}{lll}
    \rho_1 \mathbf{x_1 y_1 }  + \rho_3 ( \mathbf{x_{2n-1} y_3 } + \mathbf{x_3 y_{2n-1}} ) + \rho_{123} \sigma_{23} \mathbf{x_3 y_1} \quad \textrm{if $j=2$} \\[0.3pc]
    \rho_1 ( \mathbf{x_1 y_{2n-3} } + \mathbf{x_{2n-3} y_1} ) + \rho_3 \mathbf{x_{2n-1} y_{2n-1} } + \rho_{123} \sigma_{23} \mathbf{x_{2n-1} y_1} \quad \textrm{if $j=2n-2$} \\[0.3pc]
    \rho_1 ( \mathbf{x_1 y_{j-1} } + \mathbf{x_{j-1} y_1} ) + \rho_3 ( \mathbf{x_{2n-1} y_{j+1} } + \mathbf{x_{j+1} y_{2n-1}} ) + \rho_{123} \sigma_{23} \mathbf{x_{j+1} y_1} \quad \textrm{otherwise.}
  \end{array}
  \right.
  \end{displaymath}
  \item For $\mathbf{x_i b}$,
  \begin{displaymath}
  \mathbf{x_i b} \mapsto
  \left\{
  \begin{array}{lll}
    \sigma_1 \mathbf{x_1 y_1 }  + \sigma_3 ( \mathbf{x_3 y_{2n-1} } + \mathbf{x_{2n-1} y_3} ) + \rho_{23} \sigma_{123} \mathbf{x_1 y_3} \quad \textrm{if $i=2$} \\[0.3pc]
    \sigma_1 ( \mathbf{x_{2n-3} y_1 } + \mathbf{x_1 y_{2n-3}} ) + \sigma_3 \mathbf{x_{2n-1} y_{2n-1} } + \rho_{23} \sigma_{123} \mathbf{x_1 y_{2n-1}} \quad \textrm{if $i=2n-2$} \\[0.3pc]
    \sigma_1 ( \mathbf{x_{i-1} y_1 } + \mathbf{x_1 y_{i-1}} ) + \sigma_3 ( \mathbf{x_{i+1} y_{2n-1} } + \mathbf{x_{2n-1} y_{i+1}} ) + \rho_{23} \sigma_{123} \mathbf{x_1 y_{i+1}} \quad \textrm{otherwise.}
  \end{array}
  \right.
  \end{displaymath}
  \item $\mathbf{a b} \mapsto (\rho_1 \sigma_3 + \rho_3 \sigma_1)(\mathbf{x_1 y_{2n-1}} + \mathbf{x_{2n-1} y_1} ) + \rho_{123} \sigma_{123} \mathbf{x_1 y_1}$.
\end{itemize}
\label{thm:main}
\end{prop}

\begin{figure}
  \centering
  \includegraphics[width=1.1\textwidth]{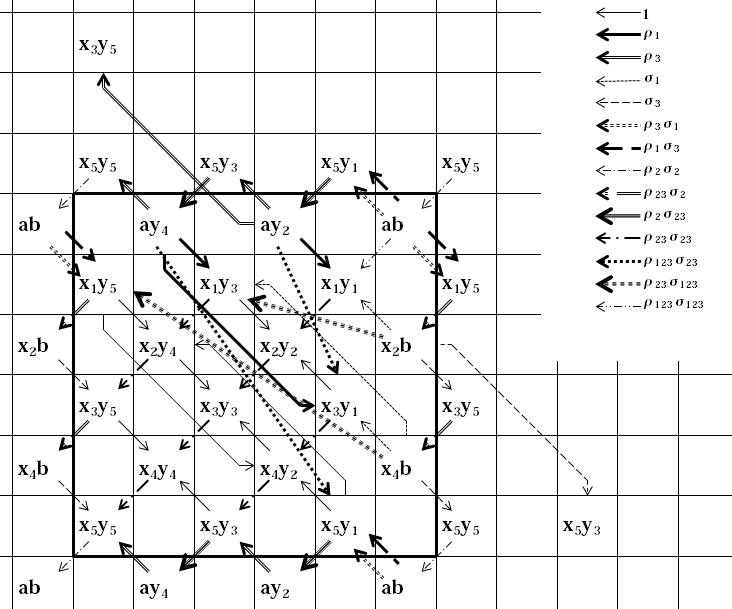}
  \caption{A diagram of $(2,6)$-torus link complement with all the differentials included.}
  \label{fig:26differential}
\end{figure}

\section{Examples}
\label{sec:example}

In this section, we will relate our result to the known calculation for knot complements and closed 3-manifolds. These examples show how to use the algebraic structure of the pairing theorem given by Lipshitz, Ozsv\'ath and Thurston in~\cite{LOT08}.

\subsection{Derived tensor product of bimodule}

The pairing of modules associated to a single boundary case is well studied by Lipshitz, Ozsv\'ath and Thurston in~\cite{LOT08}. In this section, we will be using the pairing theorem of doubly bordered cases. There are many versions of the pairing theorem depending on the types of bimodules~\cite[Theorem 2]{LOT11}, but for our purpose, the pairing of a type-$A$ module and type-$DD$ module will suffice. \\

The pairing of a doubly bordered case is also similar; the only difference is the framed arc $\mathbf{z}$. If we glue a doubly bordered diagram and a single boundary diagram together, we match the marked point $z$ from the single boundary diagram with one end of the framed arc $\mathbf{z}$. After pairing, the framed arc is reduced to a marked point on the other side of the boundary (when pairing two doubly bordered diagrams, then we connect the two framed arcs). In our example, we will be mainly interested in a type-$D$ structure obtained by the box tensor product $\widehat{CFA}( \mathcal{H}_1 ) \boxtimes \widehat{CFDD} ( \mathcal{H}_2 )$, where a single boundary diagram $\mathcal{H}_1$ is glued on the right side of a doubly bordered diagram $\mathcal{H}_2$. The resulting type-$D$ structure map $(\delta')^1$ is

\begin{displaymath}
  (\delta')^1 = \sum_{k=1}^{\infty} ((m_R)_{k+1} \otimes \mu_L   \otimes \mathbb{I}_{\widehat{CFDD}} ) (\mathbf{x} \otimes \delta^k (\mathbf{y}) ) 
\end{displaymath}
where $\mathbf{x} \in \mathfrak{S}(\mathcal{H}_1)$ and $\mathbf{y} \in \mathfrak{S}(\mathcal{H}_2)$. \\

\subsection{Infinity-surgery on right component of link}

First, we will consider an $\infty$-surgery on the right component of $(2,2n)$-torus link complement. Since the longitudes $\alpha_1^{a,L}$ and $\alpha_1^{a,R}$ of the left and right components are passing through $\beta_1$ and $\beta_2$ respectively, the $\infty$-surgery on the right components gives an unknot complement with framing $(n-1)$. We compute $\widehat{CFD}$ of the unknot complement in the following way. \\

Let $\mathcal{H}_{(2,2n)}$ be a doubly bordered diagram of $(2,2n)$ torus link complement, and $\mathcal{H}_{\infty}$ be a single bordered diagram of a solid torus. Then, the generator set $\mathfrak{S}(\mathcal{H}_{\infty} \cup_{\partial} \mathcal{H}_{(2,2n)} )$ consists of $\mathbf{w} \otimes \mathbf{a b}$ and $\mathbf{w} \otimes \mathbf{x_{2k} b}$, $k=1, \cdots, n-1$. \\

Computing $\widehat{CFA}(\mathcal{H}_{\infty})$ is easy; that is,
\begin{displaymath}
m_{k+3} (\mathbf{w}, \sigma_3, \underbrace{\sigma_{23}, \cdots, \sigma_{23}}_{k \textrm{-times}}, \sigma_2 ) = \mathbf{w}
\end{displaymath}

Now, we shall consider the type-$D$ structure of $\widehat{CFDD}(\mathcal{H}_{(2,2n)})$. We omit the terms which do not appear after taking box tensor product with $\widehat{CFA}(\mathcal{H}_{\infty})$; thus, they have no contribution in computing $\widehat{CFA}(\mathcal{H}_{\infty}) \boxtimes \widehat{CFDD}(\mathcal{H}_{(2,2n)})$.
\begin{eqnarray*}
  \delta^2 ( \mathbf{a b} ) & = & (\rho_1 \otimes \rho_{23}) \otimes (\sigma_3
  \otimes \sigma_2) \otimes \mathbf{x_2 b} + \cdots \\
  \delta^2 ( \mathbf{x_{2k} b} ) & = & (\rho_{23}) \otimes (\sigma_3
  \otimes \sigma_2) \otimes \mathbf{x_{2k+2} b} + \cdots \quad
  \textrm{for $k = 1, \cdots n-2$} \\
  \delta^2 ( \mathbf{x_{2n-2} b} ) & = & (\rho_2) \otimes (\sigma_3
  \otimes \sigma_2) \otimes \mathbf{a b} + \cdots
\end{eqnarray*}

Thus, type-$D$ structure $(\delta')^1$ is
\begin{eqnarray*}
  (\delta')^1 ( \mathbf{w} \otimes \mathbf{a b} ) & = & \mu(\rho_1 \otimes \rho_{23})
  \otimes m_3(\mathbf{w},\sigma_3, \sigma_2) \otimes \mathbf{x_2 b} \\
  & = & \rho_{123} \otimes \mathbf{w} \otimes \mathbf{x_2 b} \\
  (\delta')^1 ( \mathbf{w} \otimes \mathbf{x_{2k} b} ) & = & \mu(\rho_{23}) \otimes m_3(\mathbf{w}, \sigma_3, \sigma_2)
  \otimes \mathbf{x_{2k+2} b} \\
  & = & \rho_{23} \otimes \mathbf{w} \otimes \mathbf{x_{2k+2} b} \quad \textrm{for $k = 1, \cdots n-2$} \\
  (\delta')^1 ( \mathbf{w} \otimes \mathbf{x_{2n-2} b} ) & = & \mu(\rho_2) \otimes m_3(\mathbf{w}, \sigma_3, \sigma_2)
  \otimes \mathbf{a b} \\
  & = & \rho_2 \otimes \mathbf{w} \otimes \mathbf{a b}
\end{eqnarray*}
Compare this result with \cite[Example 2.2]{Hom}. \\

\begin{figure}
  \centering
  \includegraphics[width=1\textwidth]{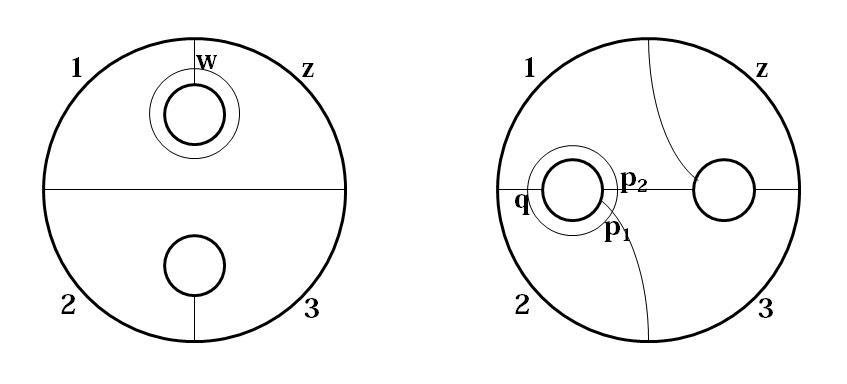}
  \caption{The diagram $\mathcal{H}_{\infty}$ on the left shows $\infty$-surgery on the right component of the link. The diagram $\mathcal{H}_{+2}$ on the right is $+2$-surgery on the right component. The $\mathcal{A}_{\infty}$ relation of $\widehat{CFA}(\mathcal{H}_{+2})$ is given as $m (q, \sigma_2) = p_1$, $m (p_1, \sigma_3, \sigma_2) = p_2$, and $m (p_2, \sigma_3, \sigma_2, \sigma_1 ) = q$. }
  \label{fig:surgery}
\end{figure}

\subsection{Knot complement of trefoil}

Consider $(2,4)$-torus link complement. If we glue the right component with a solid torus of framing $+2$, then the resulting manifold will be diffeomorphic to a trefoil complement after handleslide and blow down the +1 unknot component. A type-$D$ structure $(N_1,(\delta_1)^1):=\widehat{CFA}(\mathcal{H}_{+2}) \boxtimes \widehat{CFDD}(\mathcal{H}_{(2,4)})$ computes as

\begin{displaymath}
  \xymatrix @-1pc {
  \mathbf{p_1} \otimes \mathbf{a b} \ar@{-->}[rrrrdddd] & & \mathbf{q} \otimes
  \mathbf{x_3 y_3} \ar[ll]_{\rho_2} & & \mathbf{q} \otimes \mathbf{a y_2} \ar[ll]_{\rho_3} \ar[dd]^{\rho_1} \\
  & & & & \\
  & & & & \mathbf{q} \otimes \mathbf{x_1 y_1} \\
  & & & & \\
  & & & & \mathbf{p_2} \otimes \mathbf{a b} \ar[uu]_{\rho_{123}}
  }
\end{displaymath}

The dashed line is called an \emph{unstable chain}, where

\begin{displaymath}
  \xymatrix @-1pc {
  \cdots \rightarrow \mathbf{p_1} \otimes \mathbf{a b} \ar[rr]_{\rho_{123}} & & \mathbf{p_2} \otimes
  \mathbf{x_2 b} \ar[rr]_{\rho_{23}} \ar[dd]_{\rho_{23}} & & \mathbf{q} \otimes \mathbf{x_1 y_3} \ar[dd]_{1} \ar[rr]_{\rho_{23}} & & \mathbf{p_1}
  \otimes \mathbf{x_2 b} \ar[rr]_{\rho_2} & & \mathbf{p_2} \otimes \mathbf{a b} \rightarrow \cdots \\
  & & & & & & & & \\
  & & \mathbf{q} \otimes \mathbf{x_3 y_1} \ar[rr]_{1} & & \mathbf{q} \otimes \mathbf{x_2 y_2} & & & &
  }
\end{displaymath}

We claim that the chain complex described above is homotopy equivalent to a complex $(N_2, (\delta_2)^1)$,  which is identical to the complex above except for the unstable complex that has been replaced by
\begin{displaymath}
  \xymatrix @-1pc {
  \cdots \rightarrow \mathbf{p_1} \otimes \mathbf{a b} \ar[rr]_{\rho_{123}} & & \mathbf{p_2} \otimes
  \mathbf{x_2 b} \ar[rr]_{\rho_{23}} & & \mathbf{q} \otimes \mathbf{x_1 y_3} \ar[rr]_{\rho_{23}} & & \mathbf{p_1}
  \otimes \mathbf{x_2 b} \ar[rr]_{\rho_2} & & \mathbf{p_2} \otimes \mathbf{a b} \rightarrow \cdots \\
  }
\end{displaymath}

Define a map $\pi : N_1 \rightarrow N_2$ such that $\pi ( \mathbf{q} \otimes \mathbf{x_3 y_1} )=0$, $\pi ( \mathbf{q} \otimes \mathbf{x_2 y_2} )=0$, and otherwise identity. We also define a map $\iota : N_2 \rightarrow N_1$ as an inclusion. Then, $\pi \circ \iota = \mathbb{I}_{N_2}$ is obvious. In addition, a homotopy equivalence $H : N_1 \rightarrow N_1$ is given as
\begin{displaymath}
  H(x) := \left\{
  \begin{array}{ll}
    \mathbf{q} \otimes \mathbf{x_3 y_1} & \textrm{if $x = \mathbf{q}
    \otimes \mathbf{x_2 y_2}$} \\
    \mathbf{q} \otimes \mathbf{x_1 y_3} + \mathbf{q} \otimes
    \mathbf{x_3 y_1} & \textrm{if $x= \mathbf{q} \otimes \mathbf{x_1
    y_3}$ } \\
    \mathbf{p_2} \otimes \mathbf{x_2 b} & \textrm{if $x= \mathbf{p_2} \otimes \mathbf{x_2
    b}$ } \\
    0 & \textrm{otherwise.}
  \end{array}
  \right.
\end{displaymath}
which extends to a $\mathcal{A}(T)$-equivariant map. Then, it is clear that $\iota \circ \pi = (\delta_1)^1 \circ H + H \circ (\delta_1)^1$. \\

\begin{rem}
Compare the above result with \cite[Section 11.5]{LOT08}, from which they spelled out an algorithm to recover $\widehat{CFD}(S^3 \backslash \nu K)$ from $CFK^-$. According to their notation, the length of the unstable chain is 3 (the number of generators between two outermost ones). This length is closely related to the framing of the knot complement and concordance invariant $\tau(K)$. See \cite[Equation(11.18)]{LOT08}. In our case, the framing of the left component of the link was originally -1, but a handleslide procedure has added +4 and therefore the framing is 3. Since $\tau(\textrm{Trefoil})=1$ is less than the framing, the length of the unstable chain agrees with the framing. Interested readers will find the precise description of the relation between $\tau(K)$ and the unstable chain in~\cite[Theorem A.11]{LOT08}.
\end{rem}

\subsection{An integral surgery on Hopf link}

Hopf link is $(2,2)$-torus link. If $n_1$ and $n_2$ are two positive integers such that $n_1 n_2 \neq 1$, then $(n_1, n_2)$-surgery on Hopf link produces the lens space $L(n_1 n_2 - 1, n_1)$. The Heegaard Floer homology of the lens space has $n_1 n_2 -1$ generators whose differentials equal zero. \\

The diagram of Hopf link complement is easy. In addition, $\alpha_1^{a,L}$(respectively, $\alpha_1^{a,R}$) does not intersect $\beta_1$(respectively, $\beta_2$); therefore pairing the diagram with $\mathcal{H}_{n_1}^L$ and $\mathcal{H}_{n_2}^R$ will give a closed Heegaard diagram of the lens space $L(n_1 n_2 - 1, n_1)$. The $\mathcal{A}_{\infty}$ relation of $\widehat{CFA}(\mathcal{H}_m)$ is as follows. See Figure~\ref{fig:surgery}.
\begin{eqnarray*}
  m(q, \rho_2) & = & p_1 \\
  m(p_i, \rho_3, \underbrace{\rho_{23}, \cdots, \rho_{23}}_{j \textrm{ times}}, \rho_2 )
  & = & p_{i+j+1} \\
  m(p_m, \rho_3, \rho_2, \rho_1) & = & q
\end{eqnarray*}
$\widehat{CFDD}(S^3 \backslash \nu( \textrm{Hopf link} ))$ has two generators $\mathbf{a b}$ and $\mathbf{x_1 y_1}$. Its type-$D$ structure is given below.
\begin{eqnarray*}
  \delta^1 (\mathbf{a b}) & = & (\rho_1 \otimes \sigma_3 + \rho_3
  \otimes \sigma_1 + \rho_{123} \otimes \sigma_{123}) \otimes
  \mathbf{x_1 y_1} \\
  \delta^1 (\mathbf{x_1 y_1}) & = & \rho_2 \otimes \sigma_2 \otimes
  \mathbf{a b}
\end{eqnarray*}

\begin{rem}
See \cite[Proposition 10.1]{LOT11}. Note that Hopf link complement is $T^2 \times [0,1]$ and it is exactly an identity module described there.
\end{rem}

Let $p_i^L$ and $q^L$($p_j^R$ and $q^R$, respectively) be the generators of the bordered Heegaard diagram $\mathcal{H}_{n_1}^L$ attached to the left ($\mathcal{H}_{n_2}^R$ attached to the right, respectively). Then, we have the following $n_1 n_2 + 1$ generators of $\widehat{CFA}(\mathcal{H}_{n_1}^L) \boxtimes \widehat{CFA}(\mathcal{H}_{n_2}^R) \boxtimes \widehat{CFDD}(S^3 \backslash \nu( \textrm{Hopf link} ))$.
\begin{eqnarray*}
  p_i^L \otimes p_j^R \otimes \mathbf{a b} & \quad i = 1, \cdots,
  n_1 \textrm{ and } j = 1, \cdots, n_2 \\
  q^L \otimes q^R \otimes \mathbf{x_1 y_1}. &
\end{eqnarray*}
The only nontrivial differential is
\begin{displaymath}
  \partial^{\boxtimes} (q^L \otimes q^R \otimes \mathbf{x_1 y_1}) =
  m(q^L, \rho_2) \otimes m(q^R, \sigma_2) \otimes \mathbf{a b} =
  p_1^L \otimes p_1^R \otimes \mathbf{a b}.
\end{displaymath}
Thus, the homology of $\widehat{CFA}(\mathcal{H}_{n_1}^L) \boxtimes \widehat{CFA}(\mathcal{H}_{n_2}^R) \boxtimes \widehat{CFDD}(S^3 \backslash \nu( \textrm{Hopf link} ))$ has $n_1 n_2 -1$ generators as expected.

\section{Homotopy Equivalence}
\label{sec:homequiv}

In this section, we streamline the type-$DD$ structure computed in Section~\ref{sec:main} to a type-$DD$ structure that does not involve any differential with the algebra element 1.

\begin{figure}
  \centering
  \includegraphics[width=1.1\textwidth]{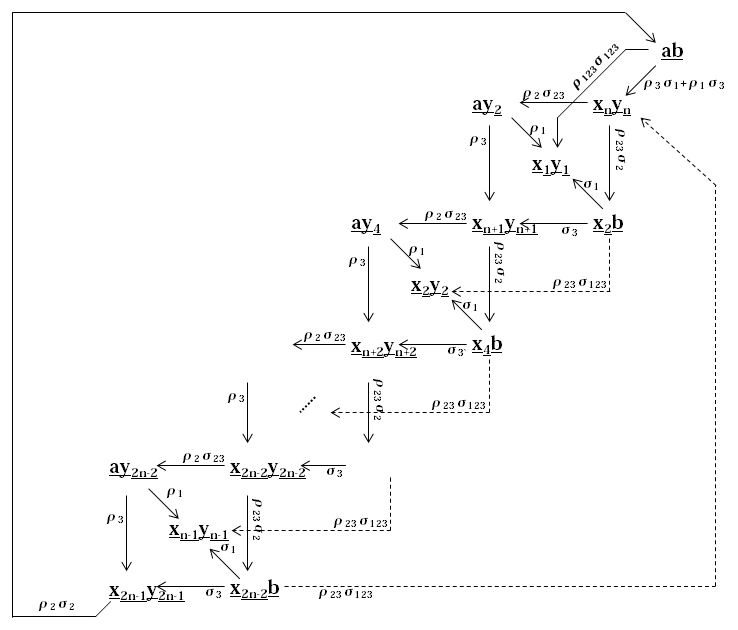}
  \caption{Simplified diagram of $\widehat{CFDD}$ of $(2,2n)$ torus link complement, $n \geq 3$. The generators are intentionally placed so that they form squares from top right to bottom left. The number of such squares is $n-1$.}
  \label{fig:simplified01}
\end{figure}

\begin{figure}
  \centering
  \includegraphics[width=1.1\textwidth]{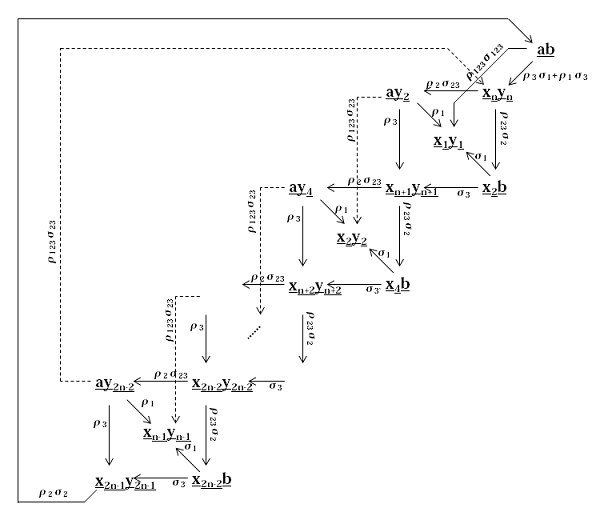}
  \caption{Another type-$DD$ structure homotopy equivalent to the original type-$DD$ structure. The differential represented by the dashed line can be changed to the differential in Figure~\ref{fig:simplified01}. }
  \label{fig:simplified02}
\end{figure}

\begin{prop}
  Type-$DD$ structure of link complement of $(2,2n)$-torus link complement, where $n \geq 3$, has the same homotopy type as the complex given in Figure~\ref{fig:simplified01}.
\end{prop}

\begin{proof}
Let $(M, \delta^1)$ denote the type-$DD$ structure computed in Proposition~\ref{thm:main} and $(N, (\delta^1)')$ the type-$DD$ structure given as Figure~\ref{fig:simplified01}. More specifically, the map $(\delta^1)'$ has the following differentials.
\begin{eqnarray*}
  \underline{\mathbf{ab}} & \mapsto & \rho_{123}\sigma_{123} \otimes \underline{\mathbf{x_1 y_1}}
  + (\rho_1 \sigma_3 + \rho_3 \sigma_1) \otimes \underline{\mathbf{x_n y_n}}, \\
  \underline{\mathbf{a y_{2k}}} & \mapsto & \rho_1 \otimes \underline{\mathbf{x_k y_k}} +
  \rho_3 \otimes \underline{\mathbf{x_{n+k} y_{n+k}}}, \ \textrm{where $k=1,\cdots,n-1$} \\
  \underline{\mathbf{x_{2k} b}} & \mapsto & \sigma_1 \otimes
  \underline{\mathbf{x_k y_k}} + \sigma_3 \otimes
  \underline{\mathbf{x_{n+k} y_{n+k}}} + \rho_{23} \sigma_{123}
  \otimes \underline{\mathbf{x_{k+1} y_{k+1}}}, \ \textrm{where $k=1,\cdots,n-1$} \\
  \underline{\mathbf{x_k y_k}} & \mapsto & 0, \ \textrm{if $k=1, \cdots, n-1$} \\
  \underline{\mathbf{x_k y_k}} & \mapsto & \rho_2 \sigma_{23} \otimes \underline{\mathbf{a y_{2(k-n+1)}}} +
  \rho_{23} \sigma_2 \otimes \underline{\mathbf{x_{2(k-n+1)} b}}, \
  \textrm{if $k=n, \cdots, 2n-2$} \\
  \underline{\mathbf{x_{2n-1} y_{2n-1}}} & \mapsto & \rho_2 \sigma_2
  \otimes \underline{\mathbf{ab}}
\end{eqnarray*}
We shall now define type-$DD$ structure maps $F:M \rightarrow \mathcal{A} (- \mathcal{Z}_L) \otimes \mathcal{A} (- \mathcal{Z}_R) \otimes N$ and $G:N \rightarrow \mathcal{A} (- \mathcal{Z}_L) \otimes \mathcal{A} (- \mathcal{Z}_R) \otimes M$. First, the map $F$ is defined as below.
\begin{eqnarray*}
  F(\mathbf{ab}) & = & \underline{ \mathbf{ab} } \\
  F(\mathbf{a y_{2k} }) & = & \underline{ \mathbf{a y_{2k}}} \\
  F(\mathbf{x_{2k} b}) & = & \underline{ \mathbf{x_{2k} b}} \\
  F(\mathbf{x_1 y_{2k-1}}) & = & \underline{ \mathbf{x_k y_k} }, \ \textrm{for $k=1, \cdots, n$} \\
  F(\mathbf{x_{2k-1} y_{2n-1}}) & = & \underline{ \mathbf{x_{k+n-1}
  y_{k+n-1} } }, \ \textrm{for $k=1, \cdots, n$} \\
  F(\mathbf{x_{2k} y_{2n-2}}) & = & \rho_2 \sigma_{23} \otimes
  \underline{\mathbf{a y_{2k}}}, \ \textrm{for $k=1, \cdots, n-1$},
\end{eqnarray*}
and zero otherwise. \\

The map $G$ is defined as follows.
\begin{eqnarray*}
  G(\underline{\mathbf{ab}}) & = & \mathbf{ab} \\
  G(\underline{\mathbf{a y_{2k}}}) & = & \mathbf{a y_{2k}} \\
  G(\underline{\mathbf{x_{2k} b}}) & = & \mathbf{x_{2k} b} \\
  G(\underline{\mathbf{x_1 y_1}}) & = & \mathbf{x_1 y_1} + \rho_{23}
  \sigma_{23} \otimes \mathbf{x_3 y_1} \\
  G(\underline{\mathbf{x_k y_k}}) & = & \mathbf{x_1 y_{2k-1}} +
  \mathbf{x_{2k-1} y_1} + \rho_{23} \sigma_{23} \otimes
  \mathbf{x_{2k+1} y_1}, \ \textrm{for $k=2, \cdots, n-1$} \\
  G(\underline{\mathbf{x_k y_k}}) & = & \mathbf{x_{2k-2n+1}
  y_{2n-1}} + \mathbf{x_{2n-1} y_{2k-2n+1}}, \ \textrm{for $k=n, \cdots,
  2n-2$} \\
  G(\underline{\mathbf{x_{2n-1} y_{2n-1}}}) & = & \mathbf{x_{2n-1}
  y_{2n-1}}
\end{eqnarray*}
These maps are easily seen satisfying the compatibility condition spelled out in \cite[Definition 2.2.55]{LOT11}. Then, the composition of two maps $F \circ G : N \rightarrow N$ is the identity map. Another composition $G \circ F$ is homotopic to identity by introducing the seemingly complicated map $H : M \rightarrow \mathcal{A} (- \mathcal{Z}_L) \otimes \mathcal{A} (- \mathcal{Z}_R) \otimes M$. For the generators of $M$ listed below, the map $H$ is defined as
\begin{eqnarray*}
  H(\mathbf{ab}) & = & 0 \\
  H(\mathbf{a y_{2k} }) & = & \rho_3 \otimes ( \mathbf{x_{2k+1}
  y_{2n-1} } + \mathbf{x_{2n-1} y_{2k+1}} ), \ \textrm{for $k=1,
  \cdots, n-2$} \\
  H(\mathbf{a y_{2n-2}}) & = & \rho_3 \otimes \mathbf{x_{2n-1}
  y_{2n- 1}} \\
  H(\mathbf{x_{2k} b}) & = & \sigma_3 \otimes ( \mathbf{x_{2n-1}
  y_{2k+1}} + \mathbf{x_{2k+1} y_{2n-1}} ), \ \textrm{for $k=1,
  \cdots, n-2$} \\
  H(\mathbf{x_{2n-2} b}) & = & \sigma_3 \otimes \mathbf{x_{2n-1}
  y_{2n-1}}
\end{eqnarray*}
Now, we need to define $H( \mathbf{x_i y_j})$. Before giving the definition, we will introduce the new notation $\mathbf{xy}(k,l) \in M$ for notational simplicity.
\begin{displaymath}
\mathbf{xy}(i,j) := \left\{ \begin{array}{ll}
  \mathbf{x_i y_j} + \mathbf{x_j y_i} & \textrm{if
  $i \neq j$} \\
  \mathbf{x_i y_j} & \textrm{if $i=j$}.
  \end{array} \right.
\end{displaymath}

\textbf{Case 1, if $i<j$}.
\begin{displaymath}
  H(\mathbf{x_i y_j}) = \left\{ \begin{array}{ll}
    \mathbf{xy}(i+1,j-1) & \textrm{if $i=1$ or $j=2n-1$} \\
    \mathbf{xy}(i+1,j-1) + \mathbf{x_{j+1} y_{i-1}} &
    \textrm{otherwise} \\
  \end{array} \right.
\end{displaymath}

\textbf{Case 2, if $i>j$}.
\begin{displaymath}
  H(\mathbf{x_i y_j}) = \left\{ \begin{array}{ll}
    \mathbf{xy}(i-1,j+1) + \rho_{23} \sigma_{23} \otimes
    \mathbf{xy}(i+1,j+1) & \textrm{if $j=1$ and $3 \leq i \leq
    2n-3$} \\
    \mathbf{xy}(i-1,j+1) & \textrm{otherwise}
  \end{array} \right.
\end{displaymath}

\textbf{Case 3, if $i=j$}.
\begin{displaymath}
  H(\mathbf{x_i y_j}) = \left\{ \begin{array}{ll}
    \rho_{23} \sigma_{23} \otimes \mathbf{x_2 y_2} & \textrm{if
    $i=j=1$} \\
    0 & \textrm{if $i=j=2n-1$} \\
    \mathbf{x_{i+1} y_{j-1}} & \textrm{otherwise}
  \end{array} \right.
\end{displaymath}

It is easy to verify that the above map satisfies $G \circ F + \mathbb{I}_M = \delta^1 \circ H + H \circ \delta^1$.
\end{proof}

\begin{rem}
  The symmetry of Figure~\ref{fig:26differential} seems to be lost after removing the differentials of the algebra element 1 since the differentials of algebra element $\rho_{23} \sigma_{123}$ are between $\underline{\mathbf{x_{2k} b}}$ and $\underline{\mathbf{x_{k+1} y_{k+1}}}$. This phenomenon is caused   because we set the map $F$ such that the bottom right corner of the original type-$DD$ structure ``collapses.'' If we set $F$ to collapse the top left corner of the original diagram, then the resulting complex will look like Figure~\ref{fig:simplified02}.
\end{rem}

%
%
%
%

\end{document}